\documentclass{amsart}

\usepackage{amsmath,latexsym,amssymb,amsthm,amsfonts,graphicx}
\usepackage{subfigure}
\usepackage{xfrac}
\usepackage{fix-cm}
\usepackage{hyperref,doi}
\usepackage{stmaryrd}
\usepackage{multirow}
\usepackage{tabularx}
\usepackage[all]{xy}
\usepackage{enumitem}
\usepackage{wrapfig}
\usepackage[title]{appendix}
\usepackage{color}

\usepackage[margin=1.25in]{geometry}

\setenumerate[1]{label=(\thesection.\arabic*)}

\numberwithin{equation}{section}
\numberwithin{figure}{section}

\newtheorem{theorem}{Theorem}[section]
\newtheorem*{maintheorem}{Main Theorem}
\newtheorem{lemma}[theorem]{Lemma}

\newtheorem{corollary}[theorem]{Corollary}
\newtheorem{proposition}[theorem]{Proposition}

\theoremstyle{definition}
\newtheorem{remark}[theorem]{Remark}
\newtheorem{remarks}[theorem]{Remarks}
\newtheorem{example}[theorem]{Example}

\newtheorem{question}[theorem]{Question}

\newtheorem{fact}[theorem]{Fact}

\def\N{\ensuremath{\mathbb{N}}}
\def\Z{\ensuremath{\mathbb{Z}}}
\def\Q{\ensuremath{\mathbb{Q}}}
\def\R{\ensuremath{\mathbb{R}}}

\def\L{\ensuremath{\mathbb{L}}}
\def\SS{\ensuremath{\mathbb{S}}}

\newcommand{\pa}[1]{\left(#1\right)}
\newcommand{\cpa}[1]{\left\{#1\right\}}
\newcommand{\wt}[1]{\widetilde{#1}}
\newcommand{\tn}[1]{\textnormal{#1}}
\newcommand{\br}[1]{\left[#1\right]}
\newcommand{\fg}[1]{\left\langle #1\right\rangle}

\newcommand{\Int}[1]{\tn{Int} \, #1}

\newcommand{\card}[1]{\left| #1 \right|}
\newcommand{\norm}[1]{\left\| #1 \right\|}
\newcommand{\la}[2]{\mathbb{L}\pa{#1,#2}}
\newcommand{\sst}[1]{\mathbb{S}\pa{#1}}
\newcommand{\ws}[2]{\mathbb{W}\pa{#1,#2}}
\newcommand{\ci}[3]{H^{#1}_{e}\pa{#2;#3}}
\newcommand{\rci}[3]{\wt{H}^{#1}_{e}\pa{#2;#3}}
\newcommand{\rc}[3]{\wt{H}^{#1}\pa{#2;#3}}

\newcommand{\zrc}[2]{\wt{H}^{#1}\pa{#2}}
\newcommand{\zrci}[2]{\wt{H}^{#1}_{e}\pa{#2}}
\newcommand{\cohom}[3]{H^{#1}\pa{#2;#3}}

\newcommand{\ilim}{\mathop{\varinjlim}\limits}
\newcommand{\csi}{\mathbin{\natural}}

\newcommand{\ec}[1]{\llbracket #1 \rrbracket}
\newcommand{\ecl}[1]{\bigl\llbracket #1 \bigr\rrbracket}
\renewcommand{\hom}[3]{\tn{Hom}_{#1}\pa{#2,#3}}

\font\cuf=cmtt8
\newcommand{\curl}[1]{{\cuf #1}}

\begin{document}
\title{Extreme Nonuniqueness of End-Sum}

\author[J.~Calcut]{Jack S. Calcut}
\address{Department of Mathematics\\
         Oberlin College\\
         Oberlin, OH 44074}
\email{jcalcut@oberlin.edu}
\urladdr{\href{http://www.oberlin.edu/faculty/jcalcut/}{\curl{http://www.oberlin.edu/faculty/jcalcut/}}}

\author[C.~Guilbault]{Craig R. Guilbault}
\address{Department of Mathematical Sciences\\
					University of Wisconsin-Milwaukee\\
					Milwaukee, WI 53201}
\email{craigg@uwm.edu}
\urladdr{\href{http://people.uwm.edu/craigg/}{\curl{http://people.uwm.edu/craigg/}}}

\author[P.~Haggerty]{Patrick V. Haggerty}
\email{patrick.v.haggerty@gmail.com}
\urladdr{\href{https://sites.google.com/view/patrickvhaggerty-math/}{\curl{https://sites.google.com/view/patrickvhaggerty-math/}}}

\keywords{End-sum, connected sum at infinity, end-cohomology, proper homotopy, direct limit, infinitely generated abelian group.}
\subjclass[2010]{Primary 57R19; Secondary 55P57.}

\dedicatory{Dedicated to the memory of Andrew Ranicki}

\date{\today}

\begin{abstract}
We give explicit examples of pairs of one-ended, open $4$-manifolds whose
end-sums yield uncountably many manifolds with distinct
proper homotopy types. This answers strongly in the affirmative a conjecture
of Siebenmann regarding the nonuniqueness of end-sums. In
addition to the construction of these examples, we provide a detailed
discussion of the tools used to distinguish them; most importantly, the end-cohomology algebra.
Key to our Main Theorem is an understanding of
this algebra for an end-sum in terms of the
algebras of the summands together with ray-fundamental classes
determined by the rays used to perform the end-sum.
Differing ray-fundamental classes allow us to
distinguish the various examples, but only through the subtle theory of
infinitely generated abelian groups. An appendix is included which contains
the necessary background from that area.
\end{abstract}

\maketitle

\section{Introduction}
\label{sec:introduction}

Our primary goal is a proof of the following theorem, which emphatically
affirms a conjecture of Siebenmann~\cite[p.~1805]{cks} addressed in an earlier article by the
first and third authors of the present paper \cite{calcuthaggerty}.

\begin{maintheorem}
There exist one-ended, open $4$-manifolds $M$ and $N$ such that the end-sum
of $M$ and $N$ yields uncountably many manifolds with distinct
proper homotopy types.
\end{maintheorem}

In addition to definitions, background, and proofs, we carefully develop
the tools needed to distinguish between the aforementioned manifolds.
Foremost among these is the end-cohomology algebra of an end-sum.
We also discuss some intriguing open questions.\\

End-sum is a technique for combining a pair of
noncompact $n$-manifolds in a manner that preserves the essential properties
of the summands. Sometimes called \emph{connected sum at infinity} in the literature,
end-sum is the natural analogue of both the classical connected sum of a
pair of $n$-manifolds and the boundary connected sum of a pair of
$n$-manifolds with boundaries. The earliest intentional use of the end-sum
operation appears to have been by Gompf~\cite{gompf83} in his work on smooth
structures on $\R^4$. Other applications to the study of exotic
$\R^4$s can be found in Bennett~\cite{bennett} and Calcut and Gompf~\cite{calcutgompf}.
End-sum has also been useful in studying contractible $n$-manifolds not homeomorphic
to $\R^n$. This is due to the fact that, unlike with classical
connected sums, the end-sum of a pair of contractible manifolds is again contractible.
For a sampling of such applications in dimension $3$,
see Myers~\cite{myers} and Tinsley and Wright~\cite{TW97};
in dimension $4$, see Calcut and Gompf~\cite{calcutgompf} and Sparks~\cite{sparks};
and in dimensions $n\geq 4$, see Calcut, King, and Siebenmann~\cite{cks}.
For ``incidental'' applications of end-sum to the study of contractible
open manifolds of dimension $n\geq4$, see Curtis and Kwun \cite{curtis-kwun} and Davis
\cite{davis}. These incidental (unintentional) applications are due to the
fact that the interior of a boundary connected sum may also be viewed as an end-sum
of the corresponding interiors.\\

Each variety of connected sum involves arbitrary choices that lead to
questions of well-definedness. For example, to perform a classical
\emph{connected sum} in the smooth category\footnote{Similar definitions,
conventions, and arguments allow for analogous connected sum operations in the
piecewise linear and topological categories.
For the sake of simplicity and focus, we will
restrict our attention to the smooth category.}, one begins with a pair of
smooth, connected, oriented $n$-manifolds, then chooses smooth $n$-balls
$B_{1}\subset\Int{M}$ and $B_{2}\subset\Int{N}$
and an orientation reversing diffeomorphism
$f:\partial B_1\to\partial B_2$. From there, one declares $M\#N$ to be the oriented manifold
$(M-\Int{B_1})\cup_{f}(N-\Int{B_2})$. Provided
$M$ and $N$ are connected, standard but nontrivial tools from differential
topology can be used to verify that, up to diffeomorphism, $M\#N$ does not
depend upon the choices made. See Kosinski~\cite[p.~90]{kosinski} for details. Note that
well-definedness fails if one omits the connectedness hypothesis or ignores orientations.\\

For smooth, oriented $n$-manifolds $M$ and $N$ with nonempty boundaries, a
\emph{boundary connected sum} is performed by choosing smooth
$(n-1)$-balls $B_1\subset\partial M$ and $B_2\subset\partial N$,
and an orientation reversing diffeomorphism $f:B_{1}\to B_{2}$.
Provided $\partial M$ and $\partial N$ are connected, an argument similar to
the one used for ordinary connected sums shows that the adjunction space
$M\cup_{f}N$ (suitably smoothed and oriented) is well-defined up to
diffeomorphism; it is sometimes denoted $M\diamond N$.
Again, see Kosinski~\cite[p.~97]{kosinski} for details.\\

An \emph{end-sum} of a pair of smooth, oriented,
noncompact \hbox{$n$-manifolds} $M$ and $N$ begins with the choice of properly
embedded rays $r\subset\Int{M}$ and $s\subset\Int{N}$
and regular neighborhoods $\nu r$ and $\nu s$ of those
rays. The regular neighborhoods are diffeomorphic to closed upper half-space $\R^n_+$,
so each has boundary diffeomorphic to $\R^{n-1}$.
Choose an orientation reversing diffeomorphism
$f:\partial \nu r\to\partial \nu s$ to obtain the end-sum defined by
$(M-\Int{\nu r})\cup_{f}(N-\Int{\nu s})$; sometimes this end-sum is
denoted informally as $M\natural N$. By an argument resembling those used
above, neither the choice of thickenings of $r$ and $s$ (that is, the regular
neighborhoods $\nu r$ and $\nu s$) nor the diffeomorphism $f$ affect the
diffeomorphism type of $M\natural N$. However, the choices of rays $r$ and
$s$ are another matter. For example, if $M$ has multiple ends, then rays in $M$
tending to different ends of $M$ can yield inequivalent end-sums, even in the
simple $n=2$ case.
For that reason, we focus on the most elusive case where $M$ and $N$ are one-ended.
The existence of knotted rays in $3$-manifolds poses problems unique
to that dimension. Indeed, Myers has exhibited an uncountable collection of
topologically distinct end-sums where both summands are $\R^3$.
So, quickly we arrive at the appropriate question:
\emph{For smooth, oriented, one-ended, open $n$-manifolds $M$ and $N$ where
$n\geq 4$, is end-sum well-defined up to diffeomorphism?}
In many cases the answer is affirmative.
For example, $\R^4\natural\R^4$ is always $\R^4$~\cite{gompf85}.
More generally, the end-sum of $n$-manifolds with Mittag-Leffler ends and $n\geq 4$
depends only on the chosen ends~\cite{calcutgompf}.
Nevertheless, Siebenmann conjectured that counterexamples should exist in general~\cite[p.~1805]{cks}.
His suspicion was confirmed by Calcut and Haggerty~\cite{calcuthaggerty} where,
for numerous pairs of smooth, one-ended, open
$4$-manifolds, it was shown that end-sums can produce non-diffeomorphic (in
fact, non-proper homotopy equivalent) manifolds. Here, we will refine the
techniques employed there to significantly extend that work.\\

As in the earlier work, the primary tool used to distinguish between various
noncompact $n$-manifolds will be their end-cohomology algebras---more
specifically it is the ring structure of that algebra that holds the key.
This is an essential point since every end-sum herein of a given pair of one-ended
manifolds has homology and cohomology groups (absolute and ``end'')
in each dimension that are isomorphic to those
of any other end-sum of the same two manifolds. To allow for differences in these
end-cohomology algebras, it will be necessary to work with manifolds that have
substantial cohomology at infinity. That leads us naturally to the
well-studied, but subtle, area of infinitely generated abelian groups. For the
benefit of the reader with limited background in that area, we have included
an appendix with key definitions and proofs of the fundamental facts used in
this paper. Capturing this subtle algebra in the form of a manifold requires
some care---most significantly, a precise description of the end-cohomology algebra
of an end-sum in terms of the end-cohomology algebras of the summands
with input from so-called ray-fundamental classes
determined by the chosen rays. We provide a careful development of this
topic, as suggested to us by Henry King.\\

Given past applications of end-sum, the following open question deserves attention.

\begin{question}
For contractible, open $n$-manifolds $M$ and $N$ of dimension $n\geq 4$,
is $M\natural N$ well-defined up to diffeomorphism or up to homeomorphism?
\end{question}

Note that, by Poincar\'{e} duality ``at the end'' (see Geoghegan~\cite[p.~361]{geoghegan}), the end-cohomology algebra of a
contractible, open $n$-manifold is isomorphic to the ordinary cohomology algebra of
$S^{n-1}$. So, the methods used in the present paper appear to be of no use in
attacking this problem.\\

Using ladders based on exotic spheres, Calcut and Gompf~\cite[Ex.~3.4(a)]{calcutgompf} gave pairs of smooth,
one-ended, open $n$-manifolds for some $n\geq 7$ whose end-sums are piecewise linearly
homeomorphic but not diffeomorphic.
It is unknown whether examples exist in dimension $n=4$ whose end-sums are homeomorphic but not diffeomorphic.
A key open question is the following (see~\cite[p.~3282]{calcuthaggerty} and~\cite[p.~1303]{calcutgompf}).
\begin{question}
Can the (oriented) end-sums of a smooth, oriented, one-ended, open $4$-manifold $M$
with a fixed oriented exotic $\R^4$ be distinct up to diffeomorphism?
\end{question}
If such examples exist, then it appears that distinguishing them will be difficult~\cite[Prop.~5.3]{calcutgompf}.\\

The outline of this paper is as follows. Section~\ref{sec:conventions} lays
out some conventions, defines end-sum, and discusses end-cohomology.
Section~\ref{sec:stringersladders} defines some manifolds (stringers, surgered stringers, and ladders)
useful for our purposes and computes their end-cohomology algebras.
Section~\ref{sec:lbos} classifies stringers, surgered stringers, and ladder manifolds based on closed surfaces.
Section~\ref{sec:cai_csi} defines ray-fundamental classes and presents a proof of an unpublished result
of Henry King that computes the end-cohomology algebra of a binary end-sum.
Section~\ref{sec:raycharclass} computes ray-fundamental classes in surgered stringers and ladders.
Section~\ref{sec:proofmainthm} proves the Main Theorem.
Appendix~\ref{sec:igagt} presents some relevant results from the theory of infinitely generated abelian groups.

\section*{Acknowledgement}

The authors are indebted to Henry King for valuable conversations and
for allowing us to include an exposition of his unpublished Theorem~\ref{thm:King}.
This research was supported in part by Simons Foundation Grant 427244, CRG.

\section{Conventions, End-sum, and End-cohomology}
\label{sec:conventions}

\subsection{Conventions}\label{ssec:conventions}

Throughout this paper, topological spaces are metrizable, separable, and locally compact.
In particular, each space has a compact exhaustion (see $\S$~\ref{ssec:cohominfty} below).
Our focus is on one-ended examples, but many of the foundational results established here apply to manifolds with arbitrarily many ends.
Unless explicitly stated otherwise, manifolds are smooth, connected, and oriented.
We follow the orientation conventions of Guillemin and Pollack~\hbox{\cite[Ch.~3]{gp}}.
In particular, the boundary $\partial M$ of a manifold $M$ is oriented by the outward normal first convention.
Let $\Int{M}$ denote the manifold interior of $M$.
A manifold without boundary is \emph{closed} if it is compact and is \emph{open} if it is noncompact.
A map of spaces is \emph{proper} provided the inverse image of each compact set is compact.
A \emph{ray} is a smooth proper embedding of the real half-line $[0,\infty)$.
The submanifold $[0,\infty)\subset \R$ is standardly oriented~\cite[Ch.~3]{gp}.
By $M\approx N$ we indicate diffeomorphic manifolds (not necessarily preserving orientation).\\

We will consider rays in manifold interiors as well as neatly embedded rays.
Recall that a manifold $A$ embedded in a manifold $B$ is said to be \emph{neatly embedded} provided $A$ is a closed subspace of $B$,  $\partial A = A \cap \partial B$, and $A$ meets $\partial B$ transversely (see Hirsch~\cite[p.~30]{hirsch} and Kosinski~\cite[pp.~27--31~\&~62]{kosinski}). The closed subspace condition is automatically satisfied by any proper embedding.
Now, let $r\subset M$ be a neatly embedded ray.
We let $\tau r\subset M$ denote a smooth closed tubular neighborhood of $r$ in $M$ as in Figure~\ref{fig:ray_nhbds} (left).
\begin{figure}[htbp!]
    \centerline{\includegraphics[scale=1.0]{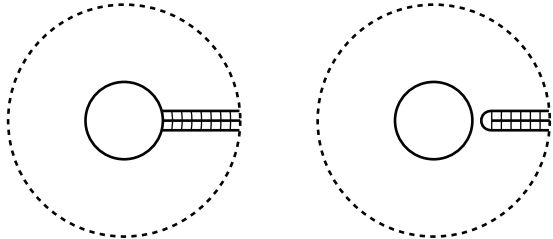}}
    \caption{Neatly embedded ray $r\subset M$ and a smooth closed tubular neighborhood $\tau r\subset M$ (left), and ray $r\subset \Int{M}$ and a smooth closed regular neighborhood $\nu r\subset\Int{M}$ (right).}
\label{fig:ray_nhbds}
\end{figure}
By definition, a \emph{closed tubular neighborhood} is a restriction of an open tubular neighborhood (see Hirsch~\cite[pp.~109--118]{hirsch} and Kosinksi~\cite[pp.~46--53]{kosinski}); we will always assume that closed tubular neighborhoods are restrictions of neat tubular neighborhoods. In particular, the disk bundle $\tau r$ over $r$ meets $\partial M$ in exactly the disk over the endpoint $0$.
Closed tubular neighborhoods of $r$ in $M$ are unique up to ambient isotopy fixing $r$.
Next, let $r\subset\Int{M}$ be a ray. We let $\nu r\subset \Int{M}$ denote a smooth closed regular neighborhood of $r$ in $\Int{M}$ as in Figure~\ref{fig:ray_nhbds} (right). Existence and ambient uniqueness of smooth closed tubular neighborhoods and collars imply the same results for smooth closed regular neighborhoods~\cite[pp.~1815]{cks}.\\

\subsection{End-sum}\label{ssec:csi}

We now define the end-sum of two noncompact manifolds.
An \emph{end-sum pair} $(M,r)$ consists of a smooth, oriented, connected, noncompact manifold $M$ together with a ray $r\subset\Int{M}$.
We allow $M$ to have arbitrarily many ends.
Consider two end-sum pairs $(M,r)$ and $(N,s)$ where $M$ and $N$ have the same dimension $m\geq 2$.
The \emph{end-sum} of $(M,r)$ and $(N,s)$, which we denote by $(M,r) \csi (N,s)$, is defined as follows.
Choose smooth closed regular neighborhoods $\nu r \subset \Int{M}$ and $\nu s \subset \Int{N}$ of $r$ and $s$ respectively.
Delete the interiors of these regular neighborhoods and glue the resulting manifolds $M - \Int \nu r$ and $N - \Int \nu s$ along their boundaries $\partial \nu r \approx \R^{m-1}$ and $\partial \nu s\approx \R^{m-1}$
by an orientation reversing diffeomorphism as in Figure~\ref{fig:csi}.
\begin{figure}[htbp!]
    \centerline{\includegraphics[scale=1.0]{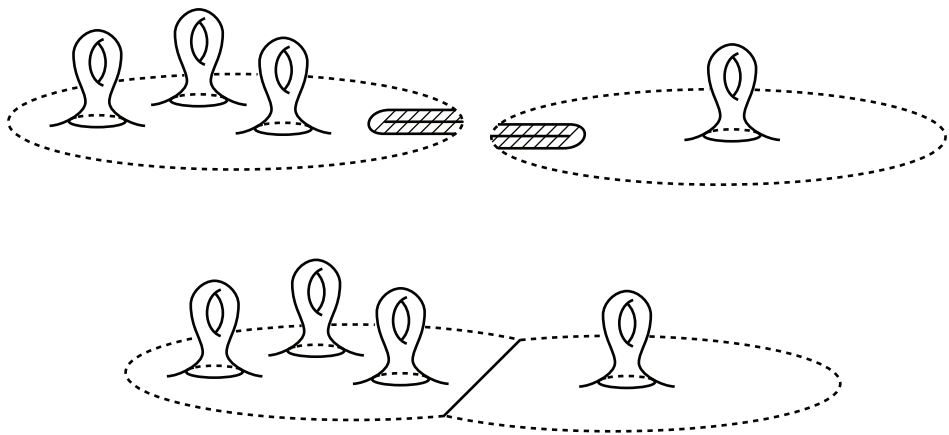}}
    \caption{End-sum of two manifold/ray pairs.}
\label{fig:csi}
\end{figure}

\begin{remark}
The manifold $(M,r) \csi (N,s)$ is smooth and oriented, and its diffeomorphism type is independent of the choices of the regular neighborhoods and the glueing diffeomorphism~\cite[$\S$~2]{calcuthaggerty}. While this binary end-sum is sufficient for our purposes, we mention that it is a special case of a more general operation that also applies to piecewise linear and topological manifolds and allows countably many summands~\cite{cks}. Alternatively, one may view the end-sum operation as the addition of a so-called $1$-handle at infinity~\cite{calcutgompf}.
\end{remark}

\subsection{End-cohomology}\label{ssec:cohominfty}

Throughout, $R$ denotes a commutative, unital ring.
We use the singular theory for ordinary (co)homology.
We suppress the coefficient ring when that ring is $\Z$.\\

We will distinguish noncompact manifolds by the isomorphism types of their (graded) end-cohomology algebras.
Just as cohomology is a homotopy invariant of spaces,
end-cohomology is a proper homotopy invariant of spaces.
We adopt the direct limit approach to end-cohomology.
An alternative may be found in several places including Conner~\cite{conner},
Raymond~\cite{raymond}, Massey~\cite[Ch.~10]{massey}, and Geoghegan~\cite[Ch.~12]{geoghegan}.
The alternative approach provides some advantages in terms of establishing the foundations
of end-cohomology and comparing it to other cohomology theories.
On the other hand, we find the direct limit approach invaluable for carrying out concrete calculations.
For the benefit of the reader---and since the arguments are straightforward and satisfying---we take the time
to develop the basics of end-cohomology straight from the direct limit definition\footnote{For background on proper homotopy,
see Guilbault~\cite[pp.~58--59]{guilbault} and Hughes and Ranicki~\cite[Ch.~3]{hr}. For background on direct systems and direct limits,
see Eilenberg and Steenrod~\cite[Ch.~8]{eilenbergsteenrod}, Massey~\cite[Appendix]{massey}, and Rotman~\cite[Ch.~6.9]{rotman}.}.\\

Fix a topological space $X$. Define the poset $\pa{\mathcal{K},\leq}$
where $\mathcal{K}$ is the set of compact subsets of $X$ and $K\leq K'$ means \mbox{$K\subseteq K'$}.
We have a direct system
of graded $R$-algebras $H^*(X-K;R)$ where $K\in\mathcal{K}$.
The morphisms of this direct system are restrictions induced by inclusions.
Define $\ci{\ast}{X}{R}$, the \emph{end-cohomology algebra}, to be the direct limit of this direct system.
For the relative version, let $(X,A)$ be a \emph{closed pair}, namely a space $X$ together with a closed subspace $A\subseteq X$.
Regard $X$ as the closed pair $(X,\emptyset)$.
Consider the direct system $H^*(X-K,A-K;R)$ where $K\in\mathcal{K}$ and the morphisms are restrictions.
Define $\ci{\ast}{X,A}{R}$ to be the direct limit of this direct system.
Similarly, reduced end-cohomology $\rci{\ast}{X,A}{R}$ is the 
direct limit of the direct system $\rc{\ast}{X-K,A-K}{R}$.\\

We employ a standard explicit model of the direct limit~\cite[p.~222]{eilenbergsteenrod} where an element of $\ci{\ast}{X,A}{R}$
is represented by an element of $H^*(X-K,A-K; R)$ for some compact $K$.
Two representatives $\alpha \in H^*(X-K,A-K; R)$ and $\alpha' \in H^*(X-K',A-K'; R)$
are equivalent if they have the same restriction in some $H^*(X-K'',A-K'';R)$, where $K,K'\subseteq K''$.\\

Recall that a \emph{compact exhaustion} of $X$ is a nested sequence
$K_1 \subseteq K_2 \subseteq \cdots$ of compact subsets of $X$ whose union equals $X$ and where
$K_j\subseteq K_{j+1}^\circ$ for each $j$.
Here, $K_{j+1}^\circ$ denotes the topological interior of $K_{j+1}$ as a subspace of $X$.
By our hypotheses on spaces, each space has a compact exhaustion (see Hocking and Young~\hbox{\cite[p.~75]{hockingyoung}}).
Let $\cpa{K_j}$ be any compact exhaustion of $X$.
As $\cpa{K_j}$ is cofinal in $\mathcal{K}$,
we may compute $\ci{\ast}{X,A}{R}$ using the direct system indexed by $\Z_{>0}$.
Namely, there is a canonical isomorphism (see~\cite[p.~224]{eilenbergsteenrod})
\begin{equation}\label{eq:direct_limit}
    \ci{\ast}{X,A}{R} \cong \ilim H^*(X - K_j,A-K_j;R)
\end{equation}
We claim that we may delete instead the topological interior $K_j^\circ$ of $K_j$ to obtain the canonical isomorphism
\begin{equation}\label{eq:direct_limit_int}
    \ci{\ast}{X,A}{R} \cong \ilim H^*(X - K_j^\circ,A-K_j^\circ;R)
\end{equation}
To prove the claim, we show that the right hand sides of~\eqref{eq:direct_limit} and~\eqref{eq:direct_limit_int}
are canonically isomorphic.
Let $G_j$ and $G'_j$ denote the $j$th terms in these direct systems.
The inclusions $K_1^\circ \subseteq K_1 \subseteq K_2^\circ \subseteq K_2 \subseteq \cdots$
induce the obvious maps between these direct systems
and give the commutative diagram
\begin{equation}\label{commdiagram}\begin{split}
\xymatrix{
G_1 \ar[r] \ar[dr] & G_2 \ar[r] \ar[dr] & G_3 \ar[r] \ar[dr] & \cdots\\
G'_1 \ar[r] \ar[u] & G'_2 \ar[r] \ar[u] & G'_3 \ar[r] \ar[u] & \cdots
}
\end{split}
\end{equation}
We get induced maps $\phi: \ilim G_j \to \ilim G'_j$ and $\psi: \ilim G'_j \to \ilim G_j$
between the direct limits~\cite[p.~223]{eilenbergsteenrod}.
It is a simple exercise to prove that $\psi\circ\phi$ and $\phi\circ\psi$ are the respective identity maps
(use Eilenberg and Steenrod~\cite[pp.~220--223]{eilenbergsteenrod}).
This proves the claim.
Passing to a subsequence in either~\eqref{eq:direct_limit} or~\eqref{eq:direct_limit_int}
canonically preserves the isomorphism type of the direct limit since these isomorphisms are independent of the
choice of compact exhaustion (see also~\cite[p.~224]{eilenbergsteenrod}).\\

A \emph{proper map of closed pairs} is a map of closed pairs $f:(X,A)\to (Y,B)$ such that $f:X\to Y$ is proper;
it follows that the restriction $\left.f\right|:A\to B$ is proper.
For example, if $(X,A)$ is a closed pair, then the inclusions
$(A,\emptyset)\hookrightarrow(X,\emptyset)$ and $(X,\emptyset)\hookrightarrow(X,A)$ are proper maps of closed pairs.
Each such map $f$ induces a morphism
\[
f^*_e : \ci{*}{Y,B}{R}\to \ci{*}{X,A}{R}
\]
Indeed, let $\cpa{L_j}$ be a compact exhaustion of $Y$.
Observe that $\cpa{K_j:=f^{-1}(L_j)}$ is a compact exhaustion of $X$.
In particular, $K_j \subseteq K_{j+1}^\circ$.
We have the commutative diagram
\begin{equation}\label{commdiagram2}\begin{split}
\xymatrix{
H^*(X-K_1,A-K_1;R)  \ar[r]  & H^*(X-K_2,A-K_2;R) \ar[r]  & \cdots \\
H^*(Y-L_1,B-L_1;R)  \ar[r] \ar[u]  & H^*(Y-L_2,B-L_2;R) \ar[r] \ar[u]  & \cdots 
}
\end{split}
\end{equation}
The rows are direct systems and the vertical maps are induced by the restrictions
\[
\left.f\right|_j : (X-K_j,A-K_j) \to (Y-L_j,B-L_j)
\]
These maps induce the morphism $f^*_e$ on the direct limits which are identified with the
respective end-cohomology algebras by~\eqref{eq:direct_limit}.
The same argument applies to reduced cohomology.
It is straightforward to verify that $\tn{id}^*_e =\tn{id}$ and
$(g\circ f)^*_e = f^*_e \circ g^*_e$.\\

\begin{lemma}\label{properhomotopic}
Let $f,g:(X,A)\to (Y,B)$ be proper maps of closed pairs.
If $f$ and $g$ are properly homotopic, then $f^*_e =g ^*_e$.
\end{lemma}

\begin{proof}
By hypothesis, there is a proper homotopy $F:X\times I \to Y$ 
such that $F_0=f$, $F_1=g$, and $F_t(A)\subseteq B$ for all $t\in I$.
Let $\tn{pr}_1 :X\times I \to X$ be projection.
Let $\cpa{L_j}$ be a compact exhaustion of $Y$.
So, $F^{-1}(L_j)\subseteq X\times I$ and
$K_j:=\tn{pr}_1 (F^{-1}(L_j))\subseteq X$ are compact.
As projection maps are open, $\cpa{K_j}$ is a 
compact exhaustion of $X$.
For each $j$, we have the restriction
\[
\left. F\right|_j : (X-K_j)\times I \to Y-L_j
\]
which is a homotopy between the restrictions
\begin{align*}
\left. f \right|_j &: X-K_j \to Y-L_j\\
\left. g \right|_j &: X-K_j \to Y-L_j
\end{align*}
Hence, $\left. f \right|_j^* =\left. g \right|_j^*$ in~\eqref{commdiagram2}.
Therefore, the induced morphisms on direct limits are equal as desired.
\end{proof}

\begin{corollary}\label{phiec}
If the closed pairs $(X,A)$ and $(Y,B)$ are proper homotopy equivalent by the proper maps
$f:(X,A)\to (Y,B)$ and $g:(Y,B)\to (X,A)$,
then $f^*_e$ and $g ^*_e$ are graded $R$-algebra isomorphisms.
\end{corollary}

\begin{proof}
By hypothesis, $g\circ f$ is proper homotopy equivalent to $\tn{id}_X$ by a
proper homotopy sending $A$ into $B$ at all times, and similarly for
$f\circ g$ and $\tn{id}_Y$.
By Lemma~\ref{properhomotopic} and the preceding observations, 
$f^*_e \circ g^*_e = \tn{id}$ and $g^*_e \circ f^*_e = \tn{id}$.
\end{proof}

\begin{lemma}\label{longexactclosedpair}
For each closed pair $(X,A)$ there is the induced long exact sequence
\[
\cdots \to \ci{k}{X,A}{R} \to \ci{k}{X}{R} \to \ci{k}{A}{R} \to \ci{k+1}{X,A}{R} \to \cdots
\]
\end{lemma}

\begin{proof}
Let $\cpa{K_j}$ be a compact exhaustion of $X$.
As $A$ is closed in $X$, $\cpa{A\cap K_j}$ is a compact exhaustion of $A$.
Consider the biinfinite commutative diagram whose $j$th column is the long exact sequence for the pair \hbox{$(X-K_j,A-K_j)$}.
The rows in this diagram are the various direct systems
$H^k(A-K_j;R)$, $H^k(X-K_j;R)$, and \hbox{$H^k(X-K_j,A-K_j;R)$}.
The maps in this diagram between successive rows induce maps of their direct limits.
The resulting sequence of direct limits is exact since the
direct limit is an exact functor in the category of $R$-modules (see~\cite[p.~225]{eilenbergsteenrod} or \cite[p.~389]{massey}).
\end{proof}

A \emph{closed triple} $(X,A,B)$ is a space $X$ together with subspaces $B\subseteq A \subseteq X$ each closed in $X$.
With the long exact sequences for the closed pairs $(A,B)$, $(X,B)$, and $(X,A)$ in hand,
a well-known diagram chase~\hbox{\cite[p.~24]{eilenbergsteenrod}} proves the following.

\begin{corollary}\label{letriple}
For each closed triple $(X,A,B)$ there is the induced long exact sequence
\[
\cdots \to \ci{k}{X,A}{R} \to \ci{k}{X,B}{R} \to \ci{k}{A,B}{R} \to \ci{k+1}{X,A}{R} \to \cdots
\]
\end{corollary}

\begin{remark}
It is crucial for end-cohomology that one consider \emph{closed} pairs and triples.
Otherwise, one does not obtain induced maps for the usual long exact sequences,
and the direct system $H^*(A\cap K_j;R)$ (where $\cpa{K_j}$ is a compact exhaustion of $X$)
need not compute $H^k_e(A;R)$.
\end{remark}

An \emph{excisive triad} $(X;A,B)$ is a space $X$ together with two closed subspaces
$A\subseteq X$ and $B\subseteq X$ such that $X=A^\circ\cup B^\circ$
where $A^\circ$ and $B^\circ$ are the topological interiors of $A$ and $B$ in $X$ respectively\footnote{Note the subtle notational distinction between a triple and a triad.}.

\begin{lemma}\label{excision}
Let $(X;A,B)$ be an excisive triad and set $C=A\cap B$.
Then, the inclusion $\phi:(A,C) \to (X,B)$ induces the \emph{excision isomorphism}
\[
\phi^* : \ci{\ast}{X,B}{R} \to \ci{\ast}{A,C}{R}
\]
\end{lemma}

\begin{proof}
Let $\cpa{K_j}$ be a compact exhaustion of $X$.
So, $\cpa{A\cap K_j}$, $\cpa{B\cap K_j}$, and $\cpa{C\cap K_j}$ are compact exhaustions
of $A$, $B$, and $C$ respectively.
We have the two direct systems
\begin{equation}
\begin{aligned}\label{twodirectsystems}
	& H^*\pa{A-K_j,C-K_j;R} \\
	& H^*\pa{X-K_j,B-K_j;R}
\end{aligned}
\end{equation}
where the morphisms in both systems are induced by inclusions.
For each $j$, we have the inclusion
\[
\phi_j :\pa{A-K_j,C-K_j} \to \pa{X-K_j,B-K_j}
\]
Observe that $X-K_j=(A-K_j)^\circ \cup (B-K_j)^\circ$ where $(A-K_j)^\circ$ denotes the
topological interior of $A-K_j$ as a subspace of $X-K_j$ and similarly for $(B-K_j)^\circ$.
Therefore, each $\phi_j^*$ is an excision isomorphism on ordinary $R$-cohomology.
By~\cite[p.~223]{eilenbergsteenrod}, these isomorphisms induce an isomorphism
between the direct limits of the direct systems~\eqref{twodirectsystems}.
Two applications of~\eqref{eq:direct_limit} now complete the proof.
\end{proof}

The following corollary is useful (compare~\cite[p.~32]{eilenbergsteenrod} and May~\cite[pp.~145]{may}).

\begin{corollary}\label{sum_pairs}
Let $(X;A,B)$ be an excisive triad and set $C=A\cap B$.
Denote the inclusion maps by $i_A:(A,C)\hookrightarrow(X,C)$ and $i_B:(B,C)\hookrightarrow(X,C)$.
Then, the map
\[
h: \ci{\ast}{X,C}{R} \to \ci{\ast}{A,C}{R} \oplus \ci{\ast}{B,C}{R}
\]
defined by $h(\alpha)=(i_A^*(\alpha),i_B^*(\alpha))$ is a graded $R$-algebra isomorphism.
\end{corollary}

Recall that the product is coordinatewise in the direct sum of algebras.

\begin{proof}
The commutative diagram of inclusions
\begin{equation}\label{eq:spaces}\begin{split}
\xymatrix{
(A,C) \ar[dd] \ar[dr]^{i_A} & & (B,C) \ar[dd] \ar[dl]_{i_B}\\
 & (X,C) \ar[dl] \ar[dr] & \\
(X,B) & & (X,A)
}
\end{split}
\end{equation}
induces the commutative diagram
\begin{equation}\label{eq:cohom_inf}\begin{split}
\xymatrix{
\ci{\ast}{A,C}{R} & & \ci{\ast}{B,C}{R}\\
 & \ci{\ast}{X,C}{R} \ar@{->>}[ul]_{i_A^*} \ar@{->>}[ur]^{i_B^*} & \\
\ci{\ast}{X,B}{R} \ar[uu]^{\cong}_{\tn{exc.}} \ar@{^{(}->}[ur] & & \ci{\ast}{X,A}{R} \ar[uu]_{\tn{exc.}}^{\cong} \ar@{_{(}->}[ul]
}
\end{split}
\end{equation}
The vertical maps are excision isomorphisms (Lemma~\ref{excision}).
Hence, the two lower maps are injective and the two upper maps are surjective.
The two diagonals are exact
being portions of long exact sequences for triples (Corollary~\ref{letriple}).
These properties of~\eqref{eq:cohom_inf} readily imply that $h$ is both injective and surjective.
\end{proof}

\begin{remark}
Let $r \subset \Int{M}$ be a ray and $\nu r\subset \Int{M}$ be a smooth closed regular neighborhood of $r$.
Define $\widehat{M}:=M-\Int{\nu r}$.
We claim that the inclusion \hbox{$\phi:(\widehat{M},\partial \nu r)\hookrightarrow(M,\nu r)$}
induces an isomorphism $\phi_e^*$ on end-cohomology.
However, the corresponding triad $(M;\widehat{M},\nu r)$ is not excisive since
$M$ is not the union of the topological interiors $\widehat{M}^\circ$ and $\nu r ^\circ$
of $\widehat{M}$ and $\nu r$ in $M$ respectively.
This nuisance is easily fixed using a closed collar.
Let $Z\approx \partial\nu r \times [0,1]$ be a closed collar on $\partial \nu r$ in $\nu r$.
Notice that $\phi$ equals the composition of the inclusions
\[
(\widehat{M},\partial \nu r)\overset{i}{\hookrightarrow}(\widehat{M}\cup Z,Z)\overset{j}{\hookrightarrow}(M,\nu r)
\]
Both induced morphisms $i_e^*$ and $j_e^*$ are isomorphisms.
The former holds since $i$ is properly homotopic to the identity map
on $(\widehat{M},\partial \nu r)$ using the obvious proper strong deformation retraction that
collapses the closed collar $Z$ to $\partial \nu r$.
The latter holds since $j_e^*$ is the excision isomorphism from the excisive triad $(M; \widehat{M}\cup Z, \nu r)$.
Hence, $\phi_e^*$ is an isomorphism and the claim is proved.
Excision is used in Section~\ref{sec:cai_csi} below
and Corollary~\ref{sum_pairs} is used in the proof of Theorem~\ref{thm:King}.
In each of these places, we leave the standard collaring fix to the reader.
\end{remark}

For a general noncompact space or manifold, it appears to be difficult
to compute the end-cohomology algebra in a comprehensible manner.
So, we deliberately construct manifolds (stringers, surgered stringers, and ladders)
with tractable algebras that fit into the following framework.\\

Let $M$ be a connected space with a compact exhaustion $\cpa{K_j}$ where $j\in\Z_{\geq 0}$.
Assume $K_0=\emptyset$.
\begin{figure}[htbp!]
    \centerline{\includegraphics[scale=1.0]{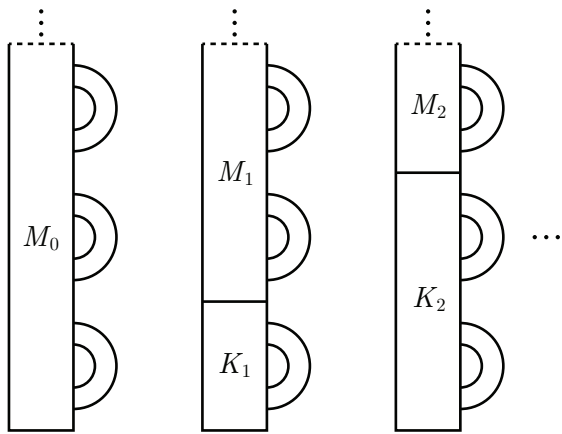}}
    \caption{Manifold $M$ with a compact exhaustion $\cpa{K_j}$ and a (closed) neighborhood system of infinity $\cpa{M_j}$.}
\label{fig:exhaustion}
\end{figure}
Define $M_j :=M-K_j^\circ$ where $K_j^\circ$ is the topological interior of $K_j$ as a subspace of $M$.
So, each $M_j$ is closed in $M$ and
\[
M=M_0 \supseteq M_1 \supseteq M_2 \supseteq \cdots
\]
is a (closed) neighborhood system of infinity as in Figure~\ref{fig:exhaustion}.
By~\eqref{eq:direct_limit_int}, we have $\ci{\ast}{M}{R} \cong \ilim H^*(M_j;R)$.\\

For each $j$, let $i_j:M_{j+1}\hookrightarrow M_j$ be the inclusion.
Suppose that for each $j\in\Z_{\geq 0}$ there is a retraction $r_j:M_j\to M_{j+1}$
(in Figure~\ref{fig:exhaustion}, the retraction $r_j$ folds up the bottom of $M_j$).
The composition $r_j\circ i_j$ equals the identity on $M_{j+1}$.
So, $i_j^* \circ r_j^*$ equals the identity on $H^*(M_{j+1};R)$ and each $i_j^*$ is surjective.
By~\cite[p.~222]{eilenbergsteenrod}, each of the canonical morphisms
\[
q_i : H^*(M_i ; R) \to \ci{*}{M}{R}
\]
is surjective with kernel $Q_i$ equal to the submodule of elements that are eventually sent to $0$
in the direct system $H^*(M_j;R)$. Here, $q_i(\alpha):= \ec{\alpha}$.
Hence, for each $i\in\Z_{\geq 0}$ we have $H^*(M_i,R)/Q_i\cong \ci{*}{M}{R}$.
This discussion applies to relative and reduced end-cohomology as well.

\section{Stringers, Surgered Stringers, and Ladders}
\label{sec:stringersladders}

In this section, we define some manifolds and present their end-cohomology algebras.
These will be used in our proof of the Main Theorem.\\

Let $X$ be a closed, connected, oriented $n$-manifold with $n\geq 2$. The \emph{stringer} based on $X$ is $[0,\infty)\times X$ with the product orientation~\cite[Ch.~3]{gp}.
Let $X_t=\cpa{t}\times X$, so the oriented boundary of the stringer is $-X_0$.
The end-cohomology algebra of the stringer is
\[
	\rci{\ast}{[0,\infty)\times X}{R} \cong \rc{\ast}{X}{R}
\]

\vspace{10pt}

The \emph{surgered stringer} $\sst{X}$ based on $X$ is obtained from the stringer on $X$ by performing countably many oriented $0$-surgeries as in Figure~\ref{fig:surgered_stringer}. 
\begin{figure}[htbp!]
    \centerline{\includegraphics[scale=1.0]{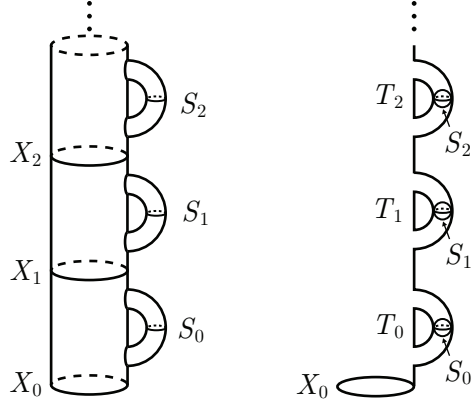}}
    \caption{Surgered stringer $\sst{X}$ and a strong deformation retract $X_0 \vee J$ of $\sst{X}$.}
\label{fig:surgered_stringer}
\end{figure}
We refer to the glued-in copies of $D^1\times S^n$ as \emph{rungs}.
The space $X_0 \vee J$ in Figure~\ref{fig:surgered_stringer} is the wedge of
$X_0$ and $J$, where $J$ is the wedge of a ray, $n$-spheres $S_j$, and $1$-spheres $T_j$.
It is a strong deformation retract of $\sst{X}$ by an argument similar to the one provided in~\cite[Lemma~3.2]{calcuthaggerty}.\\

The surgered stringer $\sst{X}$ is oriented using the orientation of the stringer $[0,\infty)\times X$.
Let $\mathbb{S}_{\br{j,k}}$ denote the points of $\sst{X}$ with heights in the interval $\br{j,k}$ as in Figure~\ref{fig:sscob}.
We orient $\mathbb{S}_{\br{j,k}}$ as a codimension-0 submanifold of $\sst{X}$.
\begin{figure}[htbp!]
    \centerline{\includegraphics[scale=1.0]{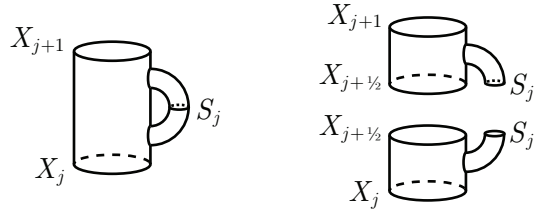}}
    \caption{Cobordisms $\mathbb{S}_{\br{j,j+1}}$, $\mathbb{S}_{\br{j,j+1/2}}$, and $\mathbb{S}_{\br{j+1/2,j+1}}$ in $\sst{X}$.}
\label{fig:sscob}
\end{figure}
We orient each $n$-sphere $S_j$ so that the oriented boundary of the cobordism $\mathbb{S}_{\br{j,j+1/2}}$ is $X_{j+1/2}-X_j+S_j$.
Thus, the oriented boundary of $\mathbb{S}_{\br{j+1/2,j+1}}$ is $X_{j+1}-X_{j+1/2}-S_j$.\\

Let $s^j$ denote the fundamental class $\br{S_j}$ of $S_j$, and let $t^j$ denote the fundamental class $\br{T_j}$ of $T_j$.
So, the nonzero reduced integer homology groups of $J$ are $\wt{H}_n(J) \cong \Z[s]$ and $\wt{H}_1(J) \cong \Z[t]$.
Define $\sigma^j$ and $\tau^j$ to be the dual fundamental classes $\br{S_j}^*$ and $\br{T_j}^*$ so that the nonzero reduced cohomology groups of $J$ are 
\begin{align*}
	\rc{n}{J}{R}	&\cong	\hom{\Z}{\Z[s]}{R} \cong R[[\sigma]]\\
	\rc{1}{J}{R}	&\cong	\hom{\Z}{\Z[t]}{R} \cong R[[\tau]]
\end{align*}
All cup products in $\rc{\ast}{J}{R}$ vanish.\\

An argument similar, but simpler, to the one provided in~\cite[$\S$~3]{calcuthaggerty} now shows that the end-cohomology algebra of $\sst{X}$ is
\[
	\rci{k}{\sst{X}}{R} \cong
	\begin{cases}
		\cohom{n}{X}{R} \oplus R[[\sigma]]/R[\sigma] &\tn{if $k = n$,}\\
		\cohom{k}{X}{R} \oplus 0  &\tn{if $2\leq k \leq n-1$,}\\
		\cohom{1}{X}{R} \oplus R[[\tau]]/R[\tau] &\tn{if $k = 1$,}\\
		0 &\tn{otherwise}
	\end{cases}
\]
The cup product is coordinatewise in the direct sum; it is that of $X$ in the first coordinate and vanishes in the second coordinate.\\

Let $X$ and $Y$ be closed, connected, oriented $n$-manifolds with $n\geq 2$.
The \emph{ladder manifold} $\la{X}{Y}$ based on $X$ and $Y$ is obtained from the stringers based on $X$ and on $Y$ by performing countably many oriented $0$-surgeries as in Figure~\ref{fig:ladder} (Ladder manifolds were the primary objects of study in~\cite{calcuthaggerty}. See that paper for more details).
\begin{figure}[htbp!]
	\centerline{\includegraphics[scale=1]{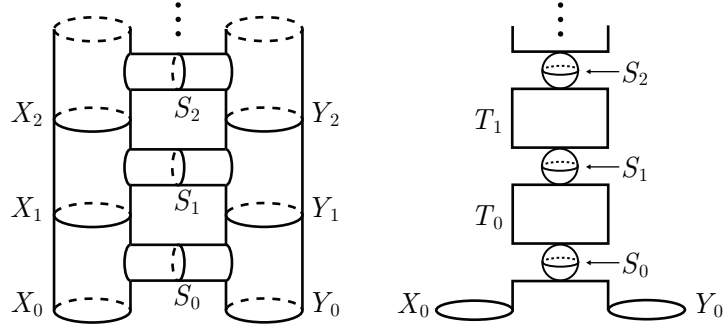}}
	\caption{Ladder manifold $\la{X}{Y}$ and a strong deformation retract $X_0 \vee J \vee Y_0$.}
	\label{fig:ladder}
\end{figure}
Again, the glued-in copies of $D^1\times S^n$ are called \emph{rungs}.
The ladder manifold is oriented using the orientations of the stringers based on $X$ and $Y$.
The oriented boundary of $\la{X}{Y}$ is $-X_0-Y_0$.
Let $\mathbb{L}_{\br{j,k}}$ denote the points of $\la{X}{Y}$ with heights in the interval $\br{j,k}$.
We orient $\mathbb{L}_{\br{j,k}}$ as a codimension-0 submanifold of $\la{X}{Y}$.
The cobordism $\mathbb{L}_{\br{j,j+1}}$ is the union of two connected cobordisms with shared boundary component $S_j$. We orient each $S_j$ so that the oriented boundaries of these cobordisms are $X_{j+1}-X_j+S_j$ and $Y_{j+1}-Y_{j}-S_j$.
The ladder manifold $\la{X}{Y}$ also contains $1$-spheres $T_j$ as shown in Figure~\ref{fig:ladder},
and it strong deformation retracts to the wedge $X_0\vee J \vee Y_0$ as explained in~\cite[p.~3287]{calcuthaggerty}.\\

Let $s^j$ denote the fundamental class $\br{S_j}$ of $S_j$, and let $t^j$ denote the fundamental class $\br{T_j}$ of $T_j$.
Again, the nonzero reduced integer homology groups of $J$ are $\wt{H}_n(J) \cong \Z[s]$ and $\wt{H}_1(J) \cong \Z[t]$.
Define $\sigma^j$ and $\tau^j$ to be the dual fundamental classes $\br{S_j}^*$ and $\br{T_j}^*$ so that the nonzero reduced cohomology groups of $J$ are 
\begin{align*}
	\rc{n}{J}{R} &\cong \hom{\Z}{\Z[s]}{R} \cong R[[\sigma]]\\
	\rc{1}{J}{R}	&\cong \hom{\Z}{\Z[t]}{R} \cong R[[\tau]]
\end{align*}
All cup products in $\rc{\ast}{J}{R}$ vanish.
By \cite[$\S$~3]{calcuthaggerty}, the end-cohomology algebra of the ladder manifold $\la{X}{Y}$ is
\[
	\rci{k}{\la{X}{Y}}{R} \cong
	\begin{cases}
		(\cohom{n}{X}{R} \oplus R[[\sigma]] \oplus \cohom{n}{Y}{R}) / K &\tn{if $k = n$,}\\
		\cohom{k}{X}{R} \oplus 0 \oplus \cohom{k}{Y}{R} &\tn{if $2\leq k \leq n-1$,}\\
		\cohom{1}{X}{R} \oplus R[[\tau]]/R[\tau] \oplus \cohom{1}{Y}{R} &\tn{if $k = 1$,}\\
		0 &\tn{otherwise}
	\end{cases}
\]
where $K := \cpa{\left. \pa{\sum \beta_i, \beta, -\sum \beta_i} \right| \beta = \sum \beta_i\sigma^i \in R[\sigma]} \cong R[\sigma]$.
The cup product is coordinatewise in the direct sum; it is that of $X$ in the first coordinate, that of $Y$ in the third coordinate, and vanishes in the second coordinate.\\

\begin{remark}\label{topdclm}
As $X$ and $Y$ are closed, connected, and oriented $n$-manifolds, we have that
\[
\rci{n}{\la{X}{Y}}{R}\cong(R\oplus R[[\sigma]] \oplus R)/K
\]
When $R=\Z$, we show in Appendix~\ref{appclm} below that the dual module of this $\Z$-module is isomorphic to $\Z$.
On the other hand, for any ring $R$ the canonical $R$-module homomorphism
\[
R\oplus R \to (R\oplus R[[\sigma]] \oplus R)/K
\]
defined by $(r,s)\mapsto \ec{(r,0,s)}$ is injective and, hence, an $R$-module isomorphism onto its image $(R\oplus 0 \oplus R)/K$.
When $R$ is a field, $(R\oplus 0 \oplus R)/K$ is a two dimensional $R$-vector space.
When $R=\Z$, $(R\oplus 0 \oplus R)/K$ is a rank two free $\Z$-module.
For any ring $R$, each cup product with value of degree $n$ must lie in $(R\oplus 0 \oplus R)/K$.
\end{remark}

For many base manifolds, surgered stringers and ladder manifolds have nonisomorphic end-cohomology algebras.
The proof of Theorem~\ref{classsurf} below shows various techniques for distinguishing these algebras.
However, in some exceptional cases these manifolds have diffeomorphic ends.

\begin{proposition}\label{ss_lm}
Let $X$ be a closed, connected, oriented $n$-manifold where $n\geq 1$.
Let $M=\la{X}{S^n}\cup_{\partial}D^{n+1}$ be the ladder manifold with the $S^{n}$ boundary component capped by an $(n+1)$-disk.
Then, $M$ is diffeomorphic to $\sst{X}$.
In particular, $\la{X}{S^n}$ and $\sst{X}$ have diffeomorphic ends and, hence, isomorphic end-cohomology algebras.
\end{proposition}

\begin{proof}
Let $N$ be the (classical) connected sum of the stringer $[0,\infty)\times X$ and countably many $(n+1)$-spheres as in Figure~\ref{fig:ladder_mod}.
\begin{figure}[htbp!]
	\centerline{\includegraphics[scale=1]{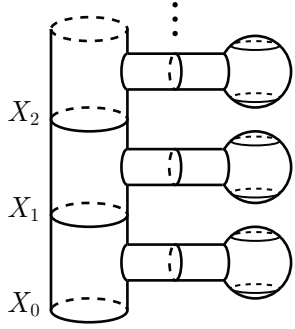}}
	\caption{Classical connected sum $N$ of the stringer based on $X$ with countably many $(n+1)$-spheres.}
	\label{fig:ladder_mod}
\end{figure}
Note that $N\approx [0,\infty)\times X$.
Performing countably many oriented $0$-surgeries on $N$ yields $M$, and performing them on $[0,\infty)\times X$ yields $\sst{X}$.
\end{proof}

\begin{remark}
Given a manifold $Y$ (not necessarily connected) and two proper disjoint rays in $Y$, \emph{ladder surgery} is the operation where one performs countably many oriented $0$-surgeries on $Y$ using the $0$-spheres given by the corresponding integer points on the rays. Properly homotopic rays yield diffeomorphic manifolds (see Definition~3.1 and Corollary~4.13 in~\cite{calcutgompf}).
\end{remark}

\begin{corollary}\label{oldconj}
For any ring $R$, there is an $R$-module isomorphism
\begin{align*}
	\textstyle f: R\oplus R[[x]]/R[x] &\longrightarrow (R\oplus R[[x]] \oplus R)/K \\
	\textstyle \pa{r,\ecl{\sum_{i=0}^{\infty} c_i x^i}} &\mapsto \textstyle \ecl{\pa{r-c_0, \sum_{i=1}^{\infty} (c_i-c_{i-1})x^i,c_0}}
\end{align*}
\end{corollary}

\begin{proof}
The topological proof of Proposition~\ref{ss_lm} determines the function $f$.
With $f$ in hand, it is straightforward to verify (purely algebraically) that $f$ is a well-defined $R$-module isomorphism.
The inverse function $f^{-1}$ is given by
\[
\textstyle \ecl{\pa{r,\sum_{i=0}^{\infty} a_i x^i,s}} \mapsto \pa{r+s,\ecl{\sum_{i=0}^\infty (s+\sum_{j=0}^i a_j)x^i}}
\]
In particular, $f$ maps  $(1,\ec{0}) \mapsto \ec{\pa{1,0,0}}$ and $\pa{1,\ecl{\frac{1}{1-x}}}\mapsto \ec{\pa{0,0,1}}$.
\end{proof}

Consider the $\Z$-module $G=\Z\oplus\Z[[x]]/\Z[x]$. The submodule $0\oplus \Z[[x]]/\Z[x]$ is determined \emph{algebraically} in an isomorphism invariant manner as the elements of $G$ sent to $0$ by every element of the dual $\Z$-module of $G$ (see Corollary~\ref{dualpowermodpoly} in Appendix~\ref{sec:igagt} below). However, the submodule $\Z\oplus 0$ cannot be determined algebraically as shown by the next corollary.

\begin{corollary}\label{cantalgdetect}
Consider the $\Z$-module $G=\Z\oplus\Z[[x]]/\Z[x]$. The elements $(1,\ec{0})$ and $ \pa{1,\ecl{\frac{1}{1-x}}}$ generate a rank two free $\Z$-submodule of $G$.
Further, there is a $\Z$-module automorphism of $G$ that interchanges these two elements.
In particular, $0\oplus \Z[[x]]/\Z[x]$ has unequal complements in $G$.
\end{corollary}

We emphasize that $G$ does \emph{not} split off $\Z\oplus\Z$ as a direct summand by Corollary~\ref{dualpowermodpoly}.

\begin{proof}
Let $f:G\to (\Z\oplus\Z[[x]]\oplus \Z)/K$ be the isomorphism from Corollary~\ref{oldconj}.
The first conclusion follows by Remark~\ref{topdclm}.
Consider an involution of the ladder manifold $\la{S^n}{S^n}$ (for example, a product of two reflections) that interchanges the stringers, reverses the orientation of each sphere $S_j$, and induces the involution $\rho$ of $(\Z\oplus \Z[[x]]\oplus \Z)/K$ given by
$\ec{(r,\gamma,s)} \mapsto \ec{(s,-\gamma,r)}$.
The automorphism $\psi:G\to G$ given by $\psi=f^{-1}\circ \rho \circ f$ interchanges the desired elements.
\end{proof}

In our proof of the Main Theorem, we will need to algebraically detect the submodule $\Z\oplus 0$ of $G$.
The previous corollary shows that this will require more of the end-cohomology algebra than just the top degree module.
We will use base manifolds $X$ with nontrivial cup products in order to algebraically detect this submodule.

\section{Stringers, Surgered Stringers, and Ladders Based on Surfaces}\label{sec:lbos}

This section classifies all stringers, surgered stringers, and ladder manifolds based on closed surfaces.
It demonstrates various methods for distinguishing end-cohomology algebras up to isomorphism.
In interesting cases, the ring structure plays the deciding role.
This classification of ladders based on surfaces answers a question raised by Calcut and Haggerty~\cite[p.~3295]{calcuthaggerty}.
In addition, its proof is good preparation for the more complicated situations that arise in subsequent sections.
Let $\Sigma_g$ denote the closed, connected, and oriented surface of genus $g\in\Z_{\geq0}$.
Throughout this section, we use integer coefficients.\\

The end-cohomology algebra of the stringer $[0,\infty)\times \Sigma_g$ is
\[
	\zrci{k}{[0,\infty)\times \Sigma_g} \cong \zrc{k}{\Sigma_g} \cong
	\begin{cases}
		\Z &\tn{if $k = 2$,}\\
		\Z^{2g} &\tn{if $k = 1$,}\\
		0 &\tn{otherwise}
	\end{cases}
\]
The cup product pairing $H^1(\Sigma_g)\times H^1(\Sigma_g)\to\Z$ is nonsingular and is given by
$\bigoplus_g
\begin{bmatrix}
0		& 1 \\
-1		& 0
\end{bmatrix}$.\\

The end-cohomology algebra of the surgered stringer $\sst{\Sigma_g}$ is
\[
	\zrci{k}{\sst{\Sigma_g}} \cong
	\begin{cases}
		\Z \oplus \Z[[\sigma]]/\Z[\sigma] &\tn{if $k = 2$,}\\
		\Z^{2g} \oplus \Z[[\tau]]/\Z[\tau] &\tn{if $k = 1$,}\\
		0 &\tn{otherwise}
	\end{cases}
\]
The cup product is coordinatewise in the direct sum, vanishes in the second coordinate, and is that of the cohomology ring of $\Sigma_g$ in the first coordinate.\\

Given $g_1,g_2\in\Z_{\geq0}$, the end-cohomology algebra of the ladder manifold $\la{\Sigma_{g_1}}{\Sigma_{g_2}}$ is
\[
	\zrci{k}{\la{\Sigma_{g_1}}{\Sigma_{g_2}}} \cong
	\begin{cases}
		(\Z \oplus \Z[[\sigma]] \oplus \Z) / K &\tn{if $k = 2$,}\\
		\Z^{2g_1} \oplus \Z[[\tau]]/\Z[\tau] \oplus \Z^{2g_2} &\tn{if $k = 1$,}\\
		0 &\tn{otherwise}
	\end{cases}
\]
where $K = \{\left.(\sum \beta_i, \beta, -\sum \beta_i) \right| \beta = \sum \beta_i\sigma^i \in \Z[\sigma]\} \cong \Z[\sigma]$.
The cup product is coordinatewise in the direct sum and vanishes in the middle coordinate.
Define the matrices
\[
C=
\begin{bmatrix}
0		& \ec{(1,0,0)} \\
\ec{-(1,0,0)}		& 0
\end{bmatrix}
\quad
D=
\begin{bmatrix}
0		& \ec{(0,0,1)} \\
\ec{-(0,0,1)}		& 0
\end{bmatrix}
\]
where $\ec{\alpha}$ is the class of $\alpha$ in $(\Z \oplus \Z[[\sigma]] \oplus \Z) / K$.
For degree one elements, the cup product in the first coordinate is given by $\bigoplus_{g_1} C$,
and in the third coordinate by $\bigoplus_{g_2} D$.\\

Of course, all of these manifolds may be capped with compact $3$-manifolds (handlebodies, for example)
to eliminate boundary and obtain open, one-ended $3$-manifolds.
However, compact caps will not alter the isomorphism types of their graded end-cohomology algebras (which is our focus).
So, we choose to work with the non-capped manifolds. We will use the following basic fact.

\begin{lemma}\label{ranklemma}
Let $F$ be a free $\Z$-module of finite rank. Let $G$ and $H$ be submodules of $F$. Then, $\tn{rank}\pa{G\cap H} \geq \tn{rank}\pa{G} + \tn{rank}\pa{H} -\tn{rank}\pa{F}$.
\end{lemma}

\begin{proof}
The hypotheses imply that $G$, $H$, and $G+H$ are free $\Z$-modules of rank at most $\tn{rank}\pa{F}$~\cite[p.~460]{dummitfoote}.
We have the exact sequence of free $\Z$-modules
\[
	0 \to G\cap H \to G\oplus H \to G+H \to 0
\]
where the second map is $g\mapsto (g,-g)$ and the third map is $(g,h)\mapsto g+h$.
Recall two facts: (i) if $E$ is a free $\Z$-module, then $\tn{rank}\pa{E} =\tn{dim}_{\Q}\pa{E\otimes_{\Z} \Q}$~\cite[pp.~373~\&~471]{dummitfoote},
and (ii) tensoring with $\Q$ is an exact functor (since $\Q$ is a flat $\Z$-module~\cite[p.~401]{dummitfoote}).
It follows that
\[
\tn{rank}\pa{G} +\tn{rank}\pa{H} = \tn{rank}\pa{G\cap H} + \tn{rank}\pa{G+H}  \leq \tn{rank}\pa{G\cap H} + \tn{rank}\pa{F}
\]
as desired.
\end{proof}

Now, we will classify up to isomorphism the algebras listed for the three types of manifolds: stringers, surgered stringers, and ladder manifolds based on surfaces.
The classification of these manifolds up to various types of equivalence will then readily follow.
Plainly, $\la{X}{Y}\approx\la{Y}{X}$ for any manifolds $X$ and $Y$.

\begin{theorem}\label{classsurf}
Two of the algebras listed are isomorphic if and only if their corresponding manifolds have the same type and are based on surfaces of equal genus, with the exception:
for each $g\in\Z_{\geq0}$ the algebras for $\sst{\Sigma_g}$, $\la{\Sigma_g}{\Sigma_0}$, and $\la{\Sigma_0}{\Sigma_g}$ are isomorphic.
In particular, the algebras for $\la{\Sigma_{g_1}}{\Sigma_{g_2}}$ and $\la{\Sigma_{g_3}}{\Sigma_{g_4}}$ are isomorphic if and only if $\cpa{g_1,g_2}=\cpa{g_3,g_4}$.
\end{theorem}

\begin{proof}
For stringers based on surfaces with unequal genus, the algebras are distinguished by the ranks of $\wt{H}^1_{e}$.
The algebras for a stringer and a surgered stringer or a ladder manifold are distinguished by the cardinalities of $\wt{H}^1_{e}$.
Corollary~\ref{dualpowermodpoly} implies that the algebras for surgered stringers based on surfaces with unequal genus are distinguished by the ranks of the duals of $\wt{H}^1_{e}$.\\

For each $g\in\Z_{\geq0}$, the manifolds $\sst{\Sigma_g}$ and $\la{\Sigma_g}{\Sigma_0}\approx\la{\Sigma_0}{\Sigma_g}$ have diffeomorphic ends by Proposition~\ref{ss_lm}.
So, their algebras are isomorphic.
In all other cases, the algebras for $\sst{\Sigma_g}$ and $\la{\Sigma_{g_1}}{\Sigma_{g_2}}$ are not isomorphic.
If $g_2=0$ and $g_1\neq g$ (or $g_1=0$ and $g_2\neq g$), then use the ranks of the duals of $\wt{H}^1_{e}$.
If $g_1\neq0$ and $g_2\neq0$, then use the ranks of the (degree two) subgroups generated by all cup products of degree one elements.
For $\sst{\Sigma_g}$ this rank is zero or one, and for $\la{\Sigma_{g_1}}{\Sigma_{g_2}}$ it is two (see Remark~\ref{topdclm}).\\

It remains to classify the algebras for ladder manifolds based on surfaces.
Suppose the following is an isomorphism
\[
f:\zrci{\ast}{\la{\Sigma_{g_1}}{\Sigma_{g_2}}} \to \zrci{\ast}{\la{\Sigma_{g_3}}{\Sigma_{g_4}}}
\]
Corollary~\ref{dualpowermodpoly} implies that the ranks of the duals of $\wt{H}^1_{e}$ are $2g_1+2g_2$ and $2g_3+2g_4$ respectively.
So, $g_1+g_2=g_3+g_4$. Suppose, by way of contradiction, that $\cpa{g_1,g_2}\neq\cpa{g_3,g_4}$.
Then, $g_1+g_2=g_3+g_4$ implies that some $g_i$ is strictly greater than the other three.
Without loss of generality, we have
\[
g_1 > g_3 \geq g_4 > g_2\geq0
\]

We will reach a contradiction using the ring structures.
First, we eliminate the $\Z[[\tau]]/\Z[\tau]$ summands in an isomorphism invariant manner.
Let $J$ denote the set of elements in $\zrci{1}{\la{\Sigma_{g_1}}{\Sigma_{g_2}}}$ that are sent to $0$ by every element in the dual of $\zrci{1}{\la{\Sigma_{g_1}}{\Sigma_{g_2}}}$.
Note that $J$ is a subgroup of $\zrci{1}{\la{\Sigma_{g_1}}{\Sigma_{g_2}}}$ and, in fact, is an ideal in $\zrci{\ast}{\la{\Sigma_{g_1}}{\Sigma_{g_2}}}$.
Similarly, we define the ideal $J'$ in $\zrci{\ast}{\la{\Sigma_{g_2}}{\Sigma_{g_3}}}$.
Evidently, $f(J)=J'$ and so we obtain an induced isomorphism of the quotient algebras where we mod out by $J$ and $J'$ respectively.
Corollary~\ref{dualpowermodpoly} implies that $J=0\oplus \Z[[\tau]]/\Z[\tau] \oplus 0$ (and similarly for $J'$).
Therefore, we have an isomorphism $\overline{f}:A\to B$ of the algebras
\[
	A=
	\begin{cases}
		(\Z \oplus \Z[[\sigma]] \oplus \Z) / K &\tn{if $k = 2$,}\\
		\Z^{2g_1} \oplus 0 \oplus \Z^{2g_2} &\tn{if $k = 1$,}\\
		0 &\tn{otherwise}
	\end{cases}
	\quad  B=
		\begin{cases}
		(\Z \oplus \Z[[\sigma]] \oplus \Z) / K &\tn{if $k = 2$,}\\
		\Z^{2g_3} \oplus 0 \oplus \Z^{2g_4} &\tn{if $k = 1$,}\\
		0 &\tn{otherwise}
	\end{cases}
\]
Let $V=\Z^{2g_1}\oplus 0 \oplus 0$, a rank $2g_1$ and degree one submodule of $A$.
Recalling Remark~\ref{topdclm}, cup products of elements of $V$ generate $(\Z\oplus0\oplus0)+K$, a rank one and degree two submodule of $A$.
We will show that cup products of elements of $\overline{f}(V)$ generate a rank \emph{two} and degree two submodule of $B$.
This contradiction will complete the proof.\\

Note the following facts.
For each element $0\neq\alpha\in V$, there exists $\alpha'\in V$ such that $\alpha\cup\alpha'\neq0$ (since the degree one cup product pairing for $\Sigma_{g_1}$ is nonsingular).
As $\overline{f}$ is an isomorphism, the previous fact holds for $\overline{f}(V)$ as well.
If $\gamma\in \Z^{2g_3}\oplus0\oplus0$ and $\delta$ has degree one, then $\gamma\cup\delta\in (\Z\oplus0\oplus0)+K$.
Similarly, if $\gamma\in 0\oplus0\oplus\Z^{2g_4}$ and $\delta$ has degree one, then $\gamma\cup\delta\in (0\oplus0\oplus\Z)+K$.
The last two facts hold since the cup product is coordinatewise.\\

Recalling that $g_1+g_2=g_3+g_4$ and $g_1 > g_3 \geq g_4 > g_2\geq0$, Lemma~\ref{ranklemma} implies that there exist elements
\begin{align*}
0\neq \alpha &\in \overline{f}(V) \cap (\Z^{2g_3}\oplus0\oplus0) \\
0\neq \beta &\in \overline{f}(V) \cap (0\oplus0\oplus\cap \Z^{2g_4})
\end{align*}
By the previous paragraph, there exist $\alpha',\beta'\in\overline{f}(V)$ such that
\begin{align*}
0\neq \alpha\cup\alpha' &\in (\Z\oplus0\oplus0)+K  \\
0\neq \beta\cup\beta'   &\in (0\oplus0\oplus\Z)+K
\end{align*}
By Remark~\ref{topdclm}, $(\Z\oplus0\oplus\Z)+K$ is free of rank two.
So, these two nonzero cup products generate a rank two submodule of degree two.
This contradiction completes the proof.
\end{proof}

\begin{corollary}
Ladder manifolds $\la{\Sigma_{g_1}}{\Sigma_{g_2}}$ and $\la{\Sigma_{g_3}}{\Sigma_{g_4}}$
based on surfaces of genera $g_1,g_2,g_3,g_4\in\Z_{\geq0}$ are proper homotopy equivalent
if and only if $\cpa{g_1,g_2}=\cpa{g_3,g_4}$.
Hence, the same classification holds up to homeomorphism and up to diffeomorphism.
\end{corollary}

\section{End-Cohomology Algebra of Binary End-Sum}
\label{sec:cai_csi}

We present a proof of an unpublished result of Henry King.
It computes the end-cohomology algebra of a binary end-sum
in terms of the algebras of the two summands together with certain ray-fundamental classes determined by the rays
used to perform the end-sum.\\

First, recall the analogue for classical connected sum. Consider two closed, connected, oriented $n$-manifolds $X$ and $Y$.
The reduced cohomology ring $\wt{H}^{\ast}\pa{X\# Y}$ is isomorphic to the quotient of the sum
$\wt{H}^{\ast}\pa{X}\oplus \wt{H}^{\ast}\pa{Y}$ by the principal ideal generated by $\pa{\br{X}^*,-\br{Y}^*}$ where $\br{X}^*\in H^n\pa{X}$ and $\br{Y}^*\in H^n\pa{Y}$ are the cohomology fundamental classes dual to the respective homology (orientation) fundamental classes. The cup product is coordinatewise in the sum.
For the unreduced ring $H^{\ast}\pa{X\# Y}$, let $P$ be the subring of $H^{\ast}\pa{X}\oplus H^{\ast}\pa{Y}$ consisting of all elements of positive degree and only those of degree zero of the form $(r,r)$ for $r\in \Z$. The desired ring is the quotient of $P$ by the principal ideal generated by $\pa{\br{X}^*,-\br{Y}^*}$.
One may prove these well-known facts by an argument structurally the same as our proof of Theorem~\ref{thm:King} below.
For end-sum and end-cohomology, the cohomology fundamental classes will be replaced by ray-fundamental classes that we now define.\\

Let $M$ be a smooth, connected, oriented, noncompact manifold of dimension $n+1\geq 2$ with compact (possibly empty) boundary.
Let $r\subset \Int{M}$ be a ray, and let $\nu r\subset \Int{M}$ be a smooth closed regular neighborhood of $r$ in $\Int{M}$ oriented as a codimension-0 submanifold of $M$.
Orient the hyperplane $\partial \nu r \approx \R^n$ as the boundary of $\nu r$.
We will define nonzero cohomology classes
\begin{align*}
\br{M,r}^*_e &\in \ci{n}{M,\nu r}{R} \\
\br{r}^*_e &\in \rci{n}{M}{R}
\end{align*}
called (respectively) the \emph{relative} and \emph{absolute ray-fundamental classes} determined by the ray $r$.
Our notation is chosen since, as will emerge, these elements are intimately related to classical fundamental classes of compact manifolds.\\

Recall that a Morse function $h\colon M \to \R$ is \emph{exhaustive} provided $h$ is proper and the image of $h$ is bounded below.

\begin{lemma}\label{lem:morse}
There exists an exhaustive Morse function $h\colon M \to \R$ such that: (i) $\left.h\right|r$ is projection, (ii) $\left.h\right|\nu r$ has just one critical point,
namely a global minimum in $\partial \nu r$, (iii) $h^{-1}([t,\infty))\cap(\nu r,\partial \nu r)\approx [t,\infty) \times (D^n,S^{n-1})$ for each $t\geq0$, and (iv) each $j\in \Z_{\geq 0}$ is a regular value of $h$.
\end{lemma}

\begin{proof}
By Whitney's embedding theorem, we may assume $M\subset\R^{2n+3}$ is a submanifold that is embedded as a closed subspace.
As $2n+3>3$, we may assume, by an ambient isotopy of $\R^{2n+3}$, that $r$ is straight in $\R^{2n+3}$.
Next, ambiently untwist $\nu r$ while fixing $r$.
Define $h(x) := \norm{x-p}^2 + c$ for an appropriate point $p \in \R^{2n+3}$ and $c\in \R$~\cite[p.~36]{milnor_morse}.
\end{proof}

Let $h\colon M \to \R$ be a Morse function given by Lemma~\ref{lem:morse}.
For each $j\in \Z_{\geq 0}$, define $M_j := h^{-1}([j,\infty))$ and $K_j:=h^{-1}((-\infty,j])$,
both oriented as codimension-0 submanifolds of $M$ (see Figure~\ref{fig:morse}).
\begin{figure}[htbp!]
	\centerline{\includegraphics[scale=1]{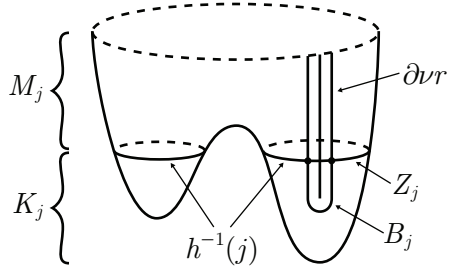}}
	\caption{Manifold $M^{n+1}$ where the Morse function $h$ is depicted as height. The $(n-1)$-sphere $\partial \widehat{Z}_j=\partial B_j$ is depicted as two dots.}
	\label{fig:morse}
\end{figure}
The $K_j$ provide a compact exhaustion of $M$.
For all sufficiently large $j$, the boundary of $M$ (compact by hypothesis) is contained in the interior of $K_j$;
without loss of generality, we assume this holds for all $j\in\Z_{\geq 0}$
(shrink $r$ towards infinity if necessary).
So, for all $j\in\Z_{\geq 0}$, $h^{-1}(j)=K_j \cap M_j$ is a finite disjoint union of closed, connected $n$-manifolds.
Let $Z_j$ be the component of $h^{-1}(j)$ that meets $\nu r$, and let $\widehat{Z}_j := Z_j - \Int{\nu r}$, both oriented as codimension-$0$ submanifolds of $\partial K_j$.
The $(n-1)$-sphere $\partial \widehat{Z}_j$ is given the boundary orientation.
Define $B_j:=\partial \nu r\cap K_j \approx D^n$ oriented as a codimension-$0$ submanifold of $\partial \nu r$.
Observe that $\partial \widehat{Z}_j=\partial B_j$ as oriented $(n-1)$-spheres.\\

For each $j\in\Z_{\geq 0}$, we define the following\footnote{Notational mnemonic: intuitively
$\widehat{X}$ is a ``nicely punctured'' copy of $X$.} (see Figure~\ref{fig:manifold_notation}).
\begin{align*}
\widehat{M}_j &:= M_j -\Int{\nu r}\\
F_j  &:= \nu r \cap M_j \approx [j,\infty)\times D^n\\
\widehat{F}_j &:=\partial \nu r \cap M_j \approx [j,\infty)\times S^{n-1}\\
\Delta_j &:= \nu r \cap Z_j \approx D^n
\end{align*}
\begin{figure}[htbp!]
	\centerline{\includegraphics[scale=1]{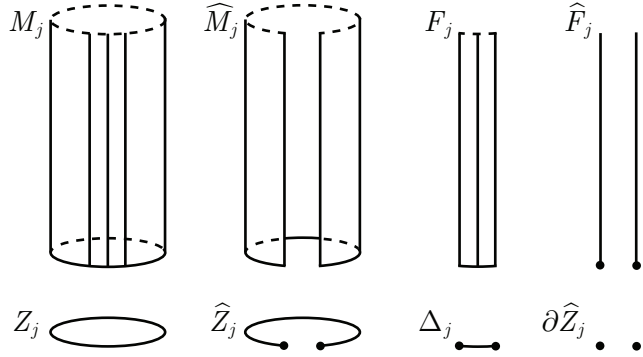}}
	\caption{Manifold $M_j$ and some relevant submanifolds.}
	\label{fig:manifold_notation}
\end{figure}
The fundamental class $\br{\partial\widehat{Z}_j}$
is our preferred generator of $H_{n-1}\pa{\widehat{F}_j}$.
By Universal Coefficients, its dual $\br{\partial\widehat{Z}_j}^*$ is our preferred generator of
\[
\hom{\Z}{H_{n-1}\pa{\widehat{F}_j}}{R} \cong \rc{n-1}{\widehat{F}_j}{R}\cong R
\]
where the latter isomorphism sends our preferred generator to  $1\in R$.
In the direct system $\rc{n-1}{\widehat{F}_j}{R}$, $j\in\Z_{\geq 0}$,
each morphism is an isomorphism carrying one preferred generator to another.
Therefore, the direct limit
\[
\rci{n-1}{\partial \nu r}{R} \cong \ilim \rc{n-1}{\widehat{F}_j}{R} \cong R
\]
has a preferred generator $\br{\partial\widehat{Z}_j}^*_e$ that is represented by each $\br{\partial\widehat{Z}_j}^*$.
In the proof of Theorem~\ref{thm:King} below, we use $\gamma_M$ to denote $\br{\partial\widehat{Z}_j}^*_e$.\\

The inclusion $\iota_j \colon \pa{\widehat{Z}_j,\partial \widehat{Z}_j} \to \pa{\widehat{M}_j, \widehat{F}_j}$ induces the following diagram, where the rows are the long exact sequences for pairs.
\begin{equation}\label{eq:coboundary}\begin{split}
\xymatrix{
	\ar[r] & \rc{n-1}{\widehat{F}_j}{R} \ar@{^{(}->}[r]^-{\delta_j} \ar[d]^{\iota_j^*}_-{\cong} & \cohom{n}{\widehat{M}_j, \widehat{F}_j}{R} \ar[d]^{\iota_j^*} \ar[r] &\\
	\ar[r] & \rc{n-1}{\partial \widehat{Z}_j}{R} \ar[r]^-{\delta'_j}_{\cong} & \cohom{n}{\widehat{Z}_j,\partial \widehat{Z}_j}{R} \ar[r] &
}
\end{split}
\end{equation}
The diagram is commutative by naturality of the coboundary map.
As $\delta'_j$ and the left $\iota_j^*$ are isomorphisms, $\delta_j$ is injective.
We have the diagram
\begin{equation}\label{eq:defchareltsj}\begin{split}
\xymatrix@R=0pt{
     H^{n-1}\pa{\widehat{F}_j;R} \ar@{^{(}->}[r]^-{\delta_j} & H^n\pa{\widehat{M}_j, \widehat{F}_j;R} & 
		H^n\pa{M_j,F_j;R} \ar[l]_-{\phi_j}^-{\cong} \ar[r]^-{\psi_j}_-{\cong} & H^n\pa{M_j;R}\\
    \br{\partial\widehat{Z_j}}^* \ar@{|-{>}}[r] & \delta_j\pa{\br{\partial\widehat{Z}_j}^*} &
		\ar@{|-{>}}[l] \br{M,r}^*_j \ar@{|-{>}}[r] & \br{r}^*_j
    }
\end{split}
\end{equation}
where $\phi_j$ is the excision isomorphism, $\psi_j$ is the isomorphism from the long exact sequence for the pair, and
$\br{M,r}^*_j$ and $\br{r}^*_j$ are \emph{defined} by the diagram.
Consider the commutative diagram $\mathcal{D}$ whose $j$th row, $j\in\Z_{\geq 0}$, equals~\eqref{eq:defchareltsj}.
The four vertical maps in $\mathcal{D}$ from row $j$ to row $j+1$ are inclusion induced.
Passing to the direct limit in $\mathcal{D}$ yields
\begin{equation}\label{eq:defcharelts}\begin{split}
\xymatrix@R=0pt{
     \rci{n-1}{\partial \nu r}{R} \ar@{^{(}->}[r]^-{\delta_M} & \ci{n}{\widehat{M}, \partial \nu r}{R} & \ci{n}{M, \nu r}{R} \ar[l]_-{\phi_M}^-{\cong} \ar[r]^-{\psi_M}_-{\cong} & \rci{n}{M}{R}\\
    \br{\partial\widehat{Z}_j}^*_e \ar@{|-{>}}[r] & \delta_M\pa{\br{\partial\widehat{Z}_j}^*_e} &
		\ar@{|-{>}}[l] \br{M,r}^*_e \ar@{|-{>}}[r] & \br{r}^*_e
    }
\end{split}
\end{equation}
where $\delta_M$ is injective, $\widehat{M}:=M - \Int \nu r$, $\phi_M$ is the excision isomorphism,
and $\psi_M$ is the isomorphism from the long exact sequence for the pair.
The relative and absolute ray-fundamental classes $\br{M,r}^*_e \in \ci{n}{M,\nu r}{R}$ and $\br{r}^*_e \in \rci{n}{M}{R}$
are \emph{defined} by~\eqref{eq:defcharelts}.\\

\begin{remarks}\label{chareltremarks}
\noindent
\begin{enumerate}[label=(\arabic*),leftmargin=*]\setcounter{enumi}{0}
\item\label{reprayclass} Let $\mathcal{D}_e$ be the diagram $\mathcal{D}$ augmented by
the direct limit row~\eqref{eq:defcharelts} together with the canonical maps in each column
from the terms in the direct system to the direct limit.
The diagram $\mathcal{D}_e$ is commutative and shows immediately that each
$\br{M,r}^*_j$ and $\br{r}^*_j$ represent (respectively) $\br{M,r}^*_e$ and $\br{r}^*_e$.
This observation holds \emph{without} any additional assumptions on $\mathcal{D}$
(such as surjectivity of the vertical maps in the last column).
\item The ray-fundamental classes are well-defined, up to isomorphism, independent of the choice of regular neighborhood $\nu r$ by uniqueness of such neighborhoods. They are also independent of the Morse function $h$ satisfying Lemma~\ref{lem:morse}. To see this, let $h'$ be another such Morse function and distinguish corresponding submanifolds of $M$ by primes.
As our Morse functions are exhaustive, each $M_j$ contains $M'_k$ for all sufficiently large $k$, and conversely.
It follows that
\[
 \ci{\ast}{M}{R} \cong \ilim H^*(M_j;R) \cong \ilim H^*(M'_j;R)
\]
and the latter of these isomorphisms carries the absolute ray-fundamental classes to one another. A similar argument applies to the relative case.
\item If $r\subset M$ is neatly embedded, then we define the ray-fundamental classes
$\br{M,r}^*_e$ and $\br{r}^*_e$ as follows. 
As in Figure~\ref{fig:morse2} (left), let $\tau r\subset M$ be a
smooth closed tubular neighborhood of $r$ in $M$
(see Section~\ref{sec:conventions} for our conventions on tubular neighborhoods). 
\begin{figure}[htbp!]
	\centerline{\includegraphics[scale=1]{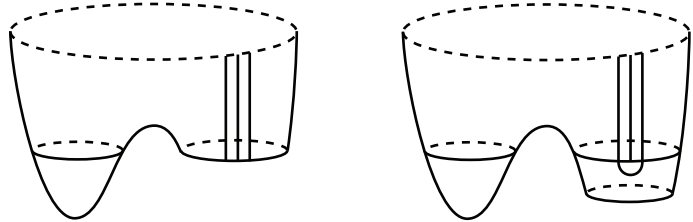}}
	\caption{Manifold $M$ containing a neatly embedded ray $r$ and a smooth closed tubular neighborhood $\tau r$ (left), and \hbox{$M'=M\cup(\tn{closed collar})$} (right).}
	\label{fig:morse2}
\end{figure}
Let $C$ be the boundary component of $M$ containing $\partial r$.
Let $M'$ equal $M$ union a closed collar on $C$. The closed collar contains an $(n+1)$-disk $B$ such that $B\cup\tau r$ is a smooth closed regular neighborhood $\nu r$ of $r$ contained in the interior of $M'$ as in Figure~\ref{fig:morse2} (right).
Evidently
\begin{align*}
\ci{\ast}{M', \nu r}{R} &\cong \ci{\ast}{M, \tau r}{R}\\
\ci{\ast}{M'}{R} &\cong \ci{\ast}{M}{R}
\end{align*}
We define the ray-fundamental classes $\br{M,r}^*_e$ and $\br{r}^*_e$
for $r\subset M$ to be the images under these isomorphisms of the ray-fundamental classes for $r\subset M'$.
\item The existence of the nonzero absolute ray-fundamental class $\br{r}^*_e$
implies that if $M$ is a smooth, oriented, connected, noncompact manifold
of dimension \hbox{$n+1\geq 2$} with compact boundary, then $\rci{\ast}{M}{R}$ is nonzero.
In particular, $R$ injects into $\ci{n}{M}{R}$.
For such a manifold $M$, $\ci{n}{M}{R}$ may indeed be the only
nonzero cohomology group in $\rci{\ast}{M}{R}$.
Consider the basic example of euclidean space.
\[
\rci{\ast}{\R^{n+1}}{R} \cong \rc{\ast}{S^n}{R} \cong \cohom{n}{S^n}{R} \cong R
\]
If $M$ has noncompact boundary, then $\rci{\ast}{M}{R}$ may vanish.
Consider the basic example of closed upper half-space $\R^{n+1}_+$ which is proper homotopy equivalent to a ray.
\[
\rci{\ast}{\R^{n+1}_+}{R} \cong \rci{\ast}{[0,\infty)}{R} \cong 0
\]
\end{enumerate}
\end{remarks}

\begin{example}\label{stringer_example}
We will compute the absolute ray-fundamental class determined by a neat straight ray in a stringer.
Fix a smooth, closed, connected, oriented $n$-manifold $Z$ where $n\geq 1$.
Let $\Delta\subset Z$ be a smoothly embedded $n$-disk, and let $z_0\in\Int{\Delta}$.
So, $r=[0,\infty)\times\cpa{z_0}$ is a neat straight ray in the stringer 
$[0,\infty)\times Z$ with smooth closed tubular neighborhood $F=[0,\infty)\times \Delta$.
We let $M_j=[j,\infty) \times Z$ and reuse the notation from Figure~\ref{fig:manifold_notation} and thereafter.\\

We have the following diagram in integer homology.
\begin{equation}\label{Xhomology}\begin{split}
\xymatrix@R=0pt{
     H_{n-1}\pa{\partial\widehat{Z}} & \ar[l]_-{\partial_*}^-{\cong}
		H_{n}\pa{\widehat{Z},\partial\widehat{Z}} \ar[r]^-{\tn{exc.}}_-{\cong} &
		H_{n}\pa{Z,\Delta}&
		\ar[l]_-{\tn{l.e.}}^-{\cong} H_{n}\pa{Z}\\
    \br{\partial\widehat{Z}} & \ar@{|-{>}}[l]  \br{\widehat{Z},\partial\widehat{Z}} \ar@{|-{>}}[r] &
		\br{Z,\Delta} & \ar@{|-{>}}[l] \br{Z}
    }
\end{split}
\end{equation}
Each of these groups is a copy of $\Z$. We claim that the preferred generators map as shown.
It is well-known that $\partial_*$ is an isomorphism here.
Seemingly less well-known is the more explicit fact that $\partial_*\pa{\br{\widehat{Z},\partial\widehat{Z}}}=\br{\partial\widehat{Z}}$ for fundamental classes and the outward normal first orientation convention; a proof appears in Kreck~\cite[Thm.~8.1]{kreck}.
A moment of reflection reveals that the second and third isomorphisms in~\eqref{Xhomology} send the preferred generators to the same generator, denoted $\br{Z,\Delta}$, of $H_{n}\pa{Z,\Delta}$ as claimed.\\

The Universal Coefficients Theorem now yields the following since all relevant Ext groups vanish.
\begin{equation}\label{Zcohomology}\begin{split}
\xymatrix@R=0pt{
     H^{n-1}\pa{\partial\widehat{Z}}  \ar[r]^-{\delta}_-{\cong} &
		H^{n}\pa{\widehat{Z},\partial\widehat{Z}} & \ar[l]_-{\tn{exc.}}^-{\cong}
		H^{n}\pa{Z,\Delta} 
		\ar[r]^-{\tn{l.e.}}_-{\cong} & H^{n}\pa{Z}\\
    \br{\partial\widehat{Z}}^*  \ar@{|-{>}}[r] & \br{\widehat{Z},\partial\widehat{Z}}^* & \ar@{|-{>}}[l] 
		\br{Z,\Delta}^*  \ar@{|-{>}}[r] & \br{Z}^*
    }
\end{split}
\end{equation}
Diagram~\eqref{Zcohomology} is canonically isomorphic to row $j=0$ in~\eqref{eq:defchareltsj}
by the obvious strong deformation retractions.
The latter diagram is canonically isomorphic 
to the direct limit diagram~\eqref{eq:defcharelts}
since every vertical map in $\mathcal{D}$ is an isomorphism.
Making the canonical identifications
\[
\ci{n}{[0,\infty)\times Z}{R} \cong H^n\pa{[0,\infty)\times Z;R} \cong H^n\pa{Z;R} \cong R
\]
with the last given by $\br{Z}^* \mapsto 1$, we have that $\br{r}^*_e=1\in R$.
This completes our example.
\end{example}

Let $(M,r)$ and $(N,s)$ be end-sum pairs (see $\S$~\ref{ssec:csi}) where $M$ and $N$ have the same dimension $n+1\geq 2$ and have compact (possibly empty) boundaries.
Let $\nu r\subset \Int{M}$ and $\nu s\subset\Int{N}$ be smooth closed regular neighborhoods of $r$ and $s$ respectively.
Let $H \subset (M,r)\csi (N,s)$ denote the image of $\partial \nu r$ (which also equals the image of $\partial \nu s$).
Let $u\subset H$ be an unknotted ray, and let $\nu u$ be a smooth closed regular neighborhood of $u$ in the interior of $S:=(M,r)\csi (N,s)$.

\begin{theorem}[H.~King]\label{thm:King}
There are isomorphisms of graded $R$-algebras
\begin{align*}
\ci{\ast}{S, \nu u}{R}    
&\cong (\ci{\ast}{M,\nu r}{R} \oplus \ci{\ast}{N, \nu s}{R})/\fg{\pa{\br{M,r}^*_e, -\br{N,s}^*_e}}\\
\rci{\ast}{S}{R}
&\cong (\rci{\ast}{M}{R} \oplus \rci{\ast}{N}{R})/\fg{\pa{\br{r}^*_e, -\br{s}^*_e}}
\end{align*}
where $\fg{\pa{\br{M,r}^*_e, -\br{N,s}^*_e}}$ and $\fg{\pa{\br{r}^*_e, -\br{s}^*_e}}$
are homogeneous principal ideals of degree $n$.
\end{theorem}
\begin{proof} 
Recall that $S$ is obtained from the disjoint union of $\widehat{M}:=M-\Int{\nu r}$ and $\widehat{N}:=N-\Int{\nu s}$ by identifying $\partial \nu r$ and $\partial\nu s$ using an orientation reversing diffeomorphism.
We have inclusions
\begin{align*}
		i_M \colon & \pa{\widehat{M}, \partial \nu r} \hookrightarrow (S,H)\\
		i_N \colon & \pa{\widehat{N}, \partial \nu s} \hookrightarrow (S,H)
\end{align*}
We orient $H$ so that $\left.i_M\right|:\partial \nu r \to H$ is an orientation preserving diffeomorphism;
it follows that $\left.i_N\right|:\partial \nu s \to H$ is an orientation reversing diffeomorphism.
Let $\omega \in \ci{n-1}{H}{R}\cong R$ be the preferred generator for this orientation.
Hence, the following hold for our preferred generators.
\begin{equation}\label{eq:genssigns}\begin{split}
\xymatrix@R=0pt{
     \rci{n-1}{H}{R} \ar[r]^-{\left.i_M\right|^*}_{\cong} & \rci{n-1}{\partial \nu r}{R} & 
		\rci{n-1}{H}{R} \ar[r]^-{\left.i_N\right|^*}_{\cong} & \rci{n-1}{\partial \nu s}{R}\\
    \omega \ar@{|-{>}}[r] & \gamma_M &
		\omega \ar@{|-{>}}[r] & -\gamma_N
    }
\end{split}
\end{equation}

The long exact sequences for the pairs give isomorphisms
\begin{align*}
		\psi_M : \ci{\ast}{M,\nu r}{R} &\overset{\cong}{\longrightarrow} \rci{\ast}{M}{R}\\
		\psi_N : \ci{\ast}{N,\nu s}{R} &\overset{\cong}{\longrightarrow} \rci{\ast}{N}{R}\\
		\psi_S : \ci{\ast}{S,\nu u}{R} &\overset{\cong}{\longrightarrow} \rci{\ast}{S}{R}
\end{align*}
Equation~\eqref{eq:defcharelts} shows that $\psi_M\pa{\br{M,r}^*_e}=\br{r}^*_e$ and 
$\psi_N\pa{\br{N,s}^*_e}=\br{s}^*_e$.
So, the reduced cohomology result will follow immediately from the relative cohomology result.
Further, the isomorphism $\psi_S$ shows that it suffices to prove the following
\[
\rci{\ast}{S}{R} \cong (\ci{\ast}{M,\nu r}{R} \oplus \ci{\ast}{N, \nu s}{R})/\fg{\pa{\br{M,r}^*_e, -\br{N,s}^*_e}}
\]

Consider the long exact sequence for the pair
\begin{equation}\label{eq:keyLE}
		\longrightarrow \rci{k-1}{H}{R} \stackrel{\delta}\longrightarrow \ci{k}{S,H}{R} \stackrel{j^*}{\longrightarrow} \rci{k}{S}{R} \longrightarrow \rci{k}{H}{R} \longrightarrow
\end{equation}
We claim that $j^*$ is an isomorphism unless $k=n$, in which case $j^*$ is surjective.
As $\rci{k}{H}{R} = 0$ for $k \neq n-1$, the claim is clear except for surjectivity of $j^*$ for $k=n-1$.
By exactness, it suffices to prove that
\[ \delta_S \colon \rci{n-1}{H}{R} \to \ci{n}{S,H}{R}\]
is injective.
The inclusions $i_M$ and $i_N$ together with naturality of the coboundary map imply the following
\begin{equation}\label{eq:comps_key}
i_M^*\circ\delta_S=\delta_M\circ \left.i_M\right|^* \qquad i_N^*\circ\delta_S=\delta_N\circ \left.i_N\right|^*
\end{equation}
Either of these equations imply that $\delta_S$ is injective since both
$\delta_M$ and $\delta_N$ are injective (see~\eqref{eq:defcharelts}) and both
$\left.i_M\right|^*$ and $\left.i_N\right|^*$ are isomorphims.
The claim is proved.\\

The claim implies that $\rci{\ast}{S}{R}$ is isomorphic to the quotient of $\ci{\ast}{S,H}{R}$ by the kernel of $j^*$.
By exactness of~\eqref{eq:keyLE}, this kernel is generated by $\delta_S(\omega)$.\\

By Corollary~\ref{sum_pairs}, the inclusions $i_M$ and $i_N$ induce the isomorphism
\[
		h \colon \ci{\ast}{S,H}{R} \overset{\cong}{\longrightarrow} \ci{\ast}{\widehat{M},\partial \nu r}{R} \oplus \ci{\ast}{\widehat{N},\partial \nu s}{R}
\]
where $h(\alpha)=\pa{i_M^{\ast}(\alpha),i_N^{\ast}(\alpha)}$. We also have the excision isomorphisms
\begin{align*}
		\phi_M: \ci{\ast}{M,\nu r}{R} &\overset{\cong}{\longrightarrow} \ci{\ast}{\widehat{M},\partial \nu r}{R}\\
		\phi_N: \ci{\ast}{N,\nu s}{R} &\overset{\cong}{\longrightarrow} \ci{\ast}{\widehat{N},\partial \nu s}{R}
\end{align*}
Therefore, the theorem will follow provided we show that the image of $\delta_S(\omega)$ under the isomorphism $h$
equals the image of $\pa{\br{M,r}^*_e,-\br{N,s}^*_e}$ under the isomorphism \hbox{$\phi_M \oplus \phi_N$}.
We have
\begin{align*}
h(\delta_S(\omega)) & = (i_M^*(\delta_S(\omega)),i_N^*(\delta_S(\omega)))\\
										& = (\delta_M\circ \left.i_M\right|^*(\omega),  \delta_N\circ \left.i_N\right|^*(\omega))\\
										& = (\delta_M(\gamma_M),\delta_N(-\gamma_N))\\
										& = \pa{\phi_M\pa{\br{M,r}^*_e},\phi_N\pa{-\br{N,s}^*_e}}
\end{align*}
where we used~\eqref{eq:comps_key}, \eqref{eq:genssigns}, and~\eqref{eq:defcharelts}.
This completes our proof of the theorem.
\end{proof}

\begin{remarks}\label{cortoHKthm}
\noindent
\begin{enumerate}[label=(\arabic*),leftmargin=*]\setcounter{enumi}{0}
\item Recall that the number of ends of a space $Y$ equals the rank of $H^0_e\pa{Y;R}$
where $R$ is a principal ideal domain~\cite[Prop.~13.4.11]{geoghegan}.
Thus, the reduced end-cohomology result in Theorem~\ref{thm:King} implies that the number of ends
(finite or infinite) of the binary end-sum $S$ equals the sum of the numbers of ends of $M$ and $N$ minus one.
In particular, if $M$ and $N$ are one-ended, then so is $S$.
\item The results in this section likely hold in the piecewise-linear and topological categories
and also for nonorientable manifolds with $R=\Z_2$.
In this paper, we will not need these generalizations.
\end{enumerate}
\end{remarks}

\section{Ray-Fundamental Classes}
\label{sec:raycharclass}

\subsection{Ray-Fundamental Classes in Ladders}
\label{ssec:fccladder}

Fix $X$ and $Y$ to be closed, connected, oriented $n$-manifolds where $n\geq2$.
Let $\L:=\la{X}{Y}$ be the ladder manifold based on $X$ and $Y$ as defined in Section~\ref{sec:stringersladders}.
Let $r$ be a ray in $\L$
emanating from $x_{0}\in X_{0}$ and intersecting each $S_{j}$ transversely as in Figure~\ref{fig:ladder_ray}.
\begin{figure}[htbp!]
	\centerline{\includegraphics[scale=1]{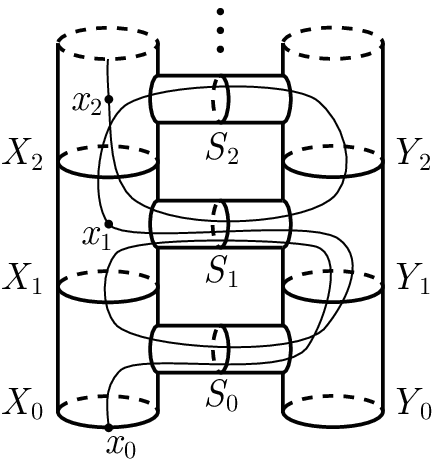}}
	\caption{Ray $r$ in ladder manifold $\L=\la{X}{Y}$.}
	\label{fig:ladder_ray}
\end{figure}
Let $F$ be a smooth closed tubular neighborhood of $r$ with a parameterization $\tau:[0,\infty)\times D^{n}\rightarrow F$ such that
$r=\tau([0,\infty)\times\overline{0})$ and, for each $j$, $F\cap S_{j}=\tau(P_{j}\times D^{n})$,
where $P_{j}$ is the (finite) set of preimages of
points where $r$ intersects $S_{j}$.
If $p\in P_{j}$, then let $D_{p}$ denote
$\tau\left(p\times D^{n}\right)$, and let $D_{0}=\tau\left(0\times D^{n}\right)=F\cap X_{0}$.
Viewing $r$ as a properly embedded oriented
submanifold of $\L$ we may consider the
$\mathbb{Z}$-intersection numbers $\varepsilon_{\mathbb{Z}}\left(r,S_{j}\right)$
(see~\cite[p.~68]{rourkesanderson} or~\cite[p.~112]{gp}). Under this convention,
$p\in P_{j}$ contributes $+1$ to $\varepsilon_{\mathbb{Z}}\left(r,S_{j}\right)$
if $r$ passes from the $X$-side of $\L$ to the $Y$-side on a small neighborhood of $p$ and it
contributes $-1$ if the reverse is true. Equivalently, give $D_{0}$ the
orientation induced by $X_{0}$ and slide that orientation along the product
structure of $F$ to orient each $D_{p}$. Then $p\in P_{j}$ contributes $+1$ to
$\varepsilon_{\mathbb{Z}}\left(r,S_{j}\right)$ if $D_{p}\hookrightarrow S_{j}$
is orientation preserving and $-1$ if $D_{p}\hookrightarrow S_{j}$ is
orientation reversing.\\

Let $\widehat{\L}=\L-F^\circ$ and $\widehat{F}=F-F^\circ=\tau([0,\infty)\times S^{n-1})$,
where $F^\circ$ denotes the topological interior of $F$ as a subspace of $\L$.
Our first goal is to understand the
coboundary map $\delta: H^{n-1}\pa{\widehat{F}}  \to H^{n}\pa{\widehat{\L},\widehat{F}}$.
To accomplish this, we use the familiar diagram
\begin{equation}\label{Lhomology}
\xymatrix@R=0pt{
     H_{n-1}\pa{\widehat{F}} & \ar[l]_-{\partial_*}
		H_{n}\pa{\widehat{\L},\widehat{F}} \ar[r]^-{\tn{exc.}}_-{\cong} &
		H_{n}\pa{\L,F}&
		\ar[l]_-{\tn{l.e.}}^-{\cong} H_{n}\pa{\L}
    }
\end{equation}
and examine the boundary map $\partial_*$.\\

By calculations in Section~\ref{sec:stringersladders},
the fundamental classes $\br{X_{0}}$, $\br{Y_{0}}$, and $\br{S_{j}}$,
$j\in\mathbb{Z}_{\geq0}$, form a free basis for $H_{n}\pa{\L}$.
By the long exact sequence for $\pa{\L,F}$ and excision,
$H_{n}\pa{\widehat{\L},\widehat{F}}$ has a free basis consisting of the relative
fundamental classes $\br{\widehat{X}_0,\partial \widehat{X}_0}$ of $\widehat{X}_{0} := X_{0}-\Int{D_{0}}$
and $\br{\widehat{S}_j,\partial \widehat{S}_j}$ of the $\widehat{S}_{j} := S_{j}-\cup_{p\in P_{j}} \Int{D_{p}}$
together with the fundamental class $\br{Y_0}$ of $Y_{0}$.
The $(n-1)$-sphere $\partial\widehat{X}_0$ is given the boundary orientation;
the fundamental class $\br{\partial\widehat{X}_0}$ 
is our preferred generator of $H_{n-1}\pa{\widehat{F}} \cong\mathbb{Z}$.
(This agrees with our orientation conventions in Section~\ref{sec:cai_csi} where $\partial \widehat{Z}_0$
played the role of $\partial \widehat{X}_0$.)
The orientation conventions established earlier in the current section imply that
$\br{\partial D_p}=-\br{\partial \widehat{X}_0}$ in $H_{n-1}\pa{\widehat{F}}$.\\

Now, $\partial_*:H_{n}\pa{\widehat{\L},\widehat{F}}  \to H_{n-1}\pa{\widehat{F}}$
is determined by its action on this basis.
We have $\partial_*\pa{\br{Y_{0}}}=0$ and
$\partial_*\pa{\br{\widehat{X}_0,\partial \widehat{X}_0}} = \br{\partial\widehat{X}_0}$
(see Example~\ref{stringer_example} above).
For each $j\in\mathbb{Z}_{\geq0}$, we have
\[
\partial_*\pa{\br{\widehat{S}_j,\partial \widehat{S}_j}} = \br{\partial\widehat{S}_j} = \sum_{p\in P_{j}} -\br{\partial D_{p}}
= \varepsilon_{\mathbb{Z}}\pa{r,S_{j}} \cdot \br{\partial\widehat{X}_0} 
\]

We now return to the pertinent coboundary map $\delta$
where we will employ the following diagram.
\begin{equation}\label{eq:boundary}\begin{split}
\xymatrix{
	H^{n-1}\pa{\widehat{F}} \ar[r]^-{\delta} \ar[d]^{h'}_{\cong} & H^{n}\pa{\widehat{\L},\widehat{F}} \ar[d]^{h}_{\cong}\\
	\hom{\Z}{H_{n-1}\pa{\widehat{F}}}{\Z} \ar[r]^-{\partial^{\ast}}  & \hom{\Z}{H_{n}\pa{\widehat{\L},\widehat{F}}}{\Z}
}
\end{split}
\end{equation}
Here $h'$ and $h$ are the surjective homomorphisms provided by
Universal Coefficients, and commutativity is verified in~\cite[p.~200]{hatcher}.
Injectivity of $h'$ and $h$ requires some specifics of the
situation at hand, but both are immediate by Universal Coefficients when
$H_{n-2}\pa{\widehat{F}}$ and $H_{n-1}\pa{\widehat{\L},\widehat{F}}$
are torsion free. That is clearly the case for $H_{n-2}\pa{\widehat{F}}$.
Next, excision and the long exact sequence
for $(\L,F)$ imply that
$H_{n-1}\pa{\widehat{\L},\widehat{F}} \cong H_{n-1}\pa{\L}$.
By calculations in Section~\ref{sec:stringersladders}, the latter is isomorphic to
$H_{n-1}\pa{X_{0}} \oplus H_{n-1}\pa{Y_{0}}$.
By Poincar\'{e} duality, $H_{n-1}\pa{X_{0}} \cong H^{1}\pa{X_{0}}$ and similarly for $Y_0$.
By Universal Coefficients, degree one $\Z$-cohomology is always torsion-free and our assertion follows.\\

The Universal Coefficients Theorem gives the following diagram dual to~\eqref{Lhomology} since all relevant Ext groups vanish.
\begin{equation}\label{Lcohomology}
\xymatrix@R=0pt{
     H^{n-1}\pa{\widehat{F}}  \ar[r]^-{\delta} &
		H^{n}\pa{\widehat{\L},\widehat{F}} & \ar[l]_-{\tn{exc.}}^-{\cong}
		H^{n}\pa{\L,F} 
		\ar[r]^-{\tn{l.e.}}_-{\cong} & H^{n}\pa{\L}
    }
\end{equation}
As in Section~\ref{sec:stringersladders}, we identify $H^{n}\pa{\L}$ with $\Z\oplus \Z[[x]]\oplus \Z$
where the dual fundamental class $\br{X_0}^*$ corresponds to the positive generator of the first summand,
$\br{S_j}^*$ corresponds to the monomial $x^j$, and $\br{Y_0}^*$ corresponds to the positive generator in the third summand.
We also identify $H^{n-1}\pa{\widehat{F}}$ with $\Z$ by $\br{\partial\widehat{X}_0}^* \leftrightarrow 1$.
Thus, the composite map $H^{n-1}\pa{\widehat{F}} \to H^{n}\pa{\L}$ may be written as
\[
\mu : \Z \to \Z \oplus \Z[[x]] \oplus \Z
\]
Define $\varepsilon_{i}=\varepsilon_{\Z}\pa{r,S_{i}}$.
With these conventions, diagram~\eqref{eq:boundary} and our description of $\partial_*$ imply that
$\mu(1) = \pa{1,\sum_{i=0}^{\infty} \varepsilon_{i}x^{i},0}$.\\

By the end of Section~\ref{ssec:cohominfty}, we have the canonical surjection
\[
q : H^{n}\pa{\L} \twoheadrightarrow H_{e}^{n}\pa{\L} 
\cong \pa{\Z \oplus \Z[[x]]\oplus \Z}/K
\]
By Remarks~\ref{chareltremarks}\ref{reprayclass}, the following is now immediate.

\begin{proposition}\label{prop:charclassinfty}
Let $r$ be a ray in $\L$ emanating from $x_{0}\in X_{0}$ and
intersecting each $S_{i}$ transversely, and let $\varepsilon_{i}=\varepsilon_{\Z}\pa{r,S_{i}}$.
Then, the absolute ray-fundamental class determined by $r$ is
\[
\textstyle \br{r}^*_e=\ecl{\pa{1,\sum_{i=0}^{\infty} \varepsilon_{i}x^{i},0}} \in \pa{\Z \oplus \Z[[x]]\oplus \Z}/K \cong
H_{e}^{n}\pa{\L}
\]
\end{proposition}

Next, we prove a simple realization theorem whose proof is reminiscent of a
Mazur-Eilenberg infinite swindle.

\begin{proposition}\label{prop:realizeclass}
If $\alpha=\sum_{i=0}^{\infty} a_i x^i \in \Z[[x]]$, then there exists a ray $r$ in $\L$ emanating from
$x_{0}\in X_{0}$ such that $\br{r}^*_e=\ec{\pa{1,\alpha,0}}$.
\end{proposition}

\begin{proof}
Recall the definition of $\mathbb{L}_{[j,k]}\subseteq \L$ in Section~\ref{sec:stringersladders}.
Let $x_{0}=\pa{0,x} \in X_{0}$ be our usual basepoint, and for each
$i\in\Z_{>0}$ choose $x_{i}=\pa{i+1/2,x} \in\mathbb{L}_{[i,i+1]}$ as in Figure~\ref{fig:ladder_ray}.
Let $r_{0}:[0,1]\rightarrow\mathbb{L}_{[0,2]}$ be a
smooth oriented path beginning at $x_{0}$, ending at $x_{1}$, and circling
through the rungs of $\mathbb{L}_{[0,2]}$ so as to realize intersection
numbers $\varepsilon_{\mathbb{Z}}\left(  r_{0},S_{0}\right)  =a_{0}$ and
$\varepsilon_{\mathbb{Z}}\left(  r_{0},S_{1}\right)  =-a_{0}$. 
With respect
to Figure~\ref{fig:ladder_ray}, this path will circle counterclockwise if $a_{0}>0$ and
clockwise if $a_0<0$; if $a_{0}=0$, then it is a vertical arc.\\

Similarly, let $r_{1}:\left[  1,2\right]  \rightarrow\mathbb{L}_{[1,3]}$ be a
path beginning at $x_{1}$, ending at $x_{2}$, and circling through the rungs
of $\mathbb{L}_{[1,3]}$ so as to realize intersection numbers $\varepsilon_{\mathbb{Z}}\left(  r_{1},S_{1}\right)  =a_{0}+a_{1}$
and $\varepsilon_{\mathbb{Z}}\left(  r_{0},S_{2}\right)  =-(a_{0}+a_{1})$. Notice that
$\varepsilon_{\mathbb{Z}}\left(  r_{0}\cup r_{1},S_{0}\right)  =a_{0}$ and
$\varepsilon_{\mathbb{Z}}\left(  r_{0}\cup r_{1},S_{1}\right)  =-a_{0}+(a_{0}+a_{1})=a_{1}$.\\

In general, choose $r_{k}:\left[  k,k+1\right]  \rightarrow\mathbb{L}_{[k,k+2]}$
beginning at $x_{k}$, ending at $x_{k+1}$, and realizing
intersection numbers $\varepsilon_{\mathbb{Z}}\left(  r_{k},S_{k}\right)  = \sum_{i=0}^k a_i$ and
$\varepsilon_{\mathbb{Z}}\left(  r_{k},S_{k+1}\right)  =-\sum_{i=0}^{k} a_i$.
Then, let $r:[0,\infty)\rightarrow \L$ be
the union of these paths, adjusted, if necessary, to make it a smooth
embedding. Choosing a nice smooth closed tubular neighborhood of $r$ and applying the proof
of Proposition~\ref{prop:charclassinfty} complete the proof.
\end{proof}

\subsection{Ray-Fundamental Classes in Surgered Stringers}
\label{ssec:fccss}

The above propositions for ladders have simpler analogues for surgered stringers.
Fix a closed oriented $n$-manifold $X$ where $n\geq 2$.
Let $\SS:=\sst{X}$ be the surgered stringer based on $X$ as defined in Section~\ref{sec:stringersladders}.
Let $r$ be a ray in $\SS$
emanating from $x_{0}\in X_{0}$ and intersecting each $S_{j}$ transversely.
Recall the definition of $\SS_{\br{j,k}}$ from Section~\ref{sec:stringersladders}.
Working as we did in ladder manifolds, we
consider the $\Z$-intersection numbers $\varepsilon_{\Z}\pa{r,S_{j}}$.
A point of $r\cap S_{j}$ at which $r$ exits $\mathbb{S}_{\br{j,j+1/2}}$
contributes $+1$, and a point where $r$ enters $\mathbb{S}_{\br{j,j+1/2}}$ contributes $-1$.\\

As before, let $F$ be a smooth closed tubular neighborhood of $r$ chosen so there
exists a parameterization $\tau:[0,\infty)\times D^{n}\rightarrow F$ with
$r=\tau([0,\infty)\times\overline{0})$ and, for each $j$, $F\cap S_{j}=\tau(P_{j}\times D^{n})$,
where $P_{j}$ is the set of preimages of $r\cap S_{j}$. 
Let $D_{0}=F\cap X_{0}$ and for each $p\in P_{j}$ let $D_{p}=\tau(p\times D^{n})\subseteq S_{j}$.
Let $\widehat{\SS}=\SS - F^\circ$; $\widehat{F}=F-F^\circ$;
$\widehat{X}_{0}:= X_{0}-\Int{D_{0}}$; and $\widehat{S}_{j}:= S_{j}-\cup_{p\in P_{j}}\Int{D_{p}}$.
Using calculations from Section~\ref{sec:stringersladders}, the long exact sequence for $\pa{\SS,F}$, excision,
and notation as above for ladders,
the relative fundamental classes
$\br{\widehat{X}_0,\partial \widehat{X}_0}$ and $\br{\widehat{S}_j,\partial \widehat{S}_j}$, $j\in\Z_{\geq 0}$,
form a free basis for
$H_{n}\pa{\widehat{\SS},\widehat{F}}$.
Our preferred generator of $H_{n-1}\pa{\widehat{F}}\cong\Z$ is $\br{\partial\widehat{X}_0}$,
and $\br{\partial D_p}=-\br{\partial \widehat{X}_0}$ in $H_{n-1}\pa{\widehat{F}}$.\\

The map $\partial_*:H_{n}\pa{\widehat{\SS},\widehat{F}}  \to H_{n-1}\pa{\widehat{F}}$
is given by $\br{\widehat{X}_0,\partial \widehat{X}_0} \mapsto \br{\partial\widehat{X}_0}$
and
$\br{\widehat{S}_j,\partial \widehat{S}_j} \mapsto \varepsilon_{\mathbb{Z}}\pa{r,S_{j}} \cdot \br{\partial\widehat{X}_0}$.
The Universal Coefficients Theorem gives the following diagram.
\begin{equation}\label{Scohomology}
\xymatrix@R=0pt{
     H^{n-1}\pa{\widehat{F}}  \ar[r]^-{\delta} &
		H^{n}\pa{\widehat{\SS},\widehat{F}} & \ar[l]_-{\tn{exc.}}^-{\cong}
		H^{n}\pa{\SS,F} 
		\ar[r]^-{\tn{l.e.}}_-{\cong} & H^{n}\pa{\SS}
    }
\end{equation}
Identify $H^{n}\pa{\SS}$ with $\Z\oplus \Z[[x]]$
by $\br{X_0}^* \leftrightarrow (1,0)$ and
$\br{S_j}^* \leftrightarrow (0,x^j)$.
Identify $H^{n-1}\pa{\widehat{F}}$ with $\Z$ by $\br{\partial\widehat{X}_0}^* \leftrightarrow 1$.
The composite map $H^{n-1}\pa{\widehat{F}} \to H^{n}\pa{\SS}$ is now written as
$ \mu : \Z \to \Z \oplus \Z[[x]] $.
Define $\varepsilon_{i}=\varepsilon_{\Z}\pa{r,S_{i}}$.
Then, $\mu(1)=\pa{1,\sum_{i=0}^{\infty} \varepsilon_i x^i}$.
We have the canonical surjection
\[
q : H^{n}\pa{\SS} \twoheadrightarrow H_{e}^{n}\pa{\SS} 
\cong \Z \oplus \Z[[x]]/\Z[x]
\]
Our work yields the following.

\begin{proposition}\label{charclassss}
Let $r$ be a ray in $\SS$ emanating from $x_{0}\in X_{0}$ and
intersecting each $S_{i}$ transversely, and let $\varepsilon_{i}=\varepsilon_{\Z}\pa{r,S_{i}}$.
Then, the absolute ray-fundamental class determined by $r$ is
\[
\textstyle \br{r}^*_e=\pa{1,\ecl{\sum_{i=0}^{\infty} \varepsilon_{i}x^{i}}} \in \Z \oplus \Z[[x]]/\Z[x] \cong
H_{e}^{n}\pa{\SS}
\]
Furthermore, if $\alpha=\sum_{i=0}^{\infty} a_i x^i \in \Z[[x]]$, then there exists a ray $r$ in $\SS$ emanating from
$x_{0}\in X_{0}$ such that $\br{r}^*_e=\pa{1,\ec{\alpha}}$.
\end{proposition}

The proof of the realization result in this proposition is simpler than that for ladder manifolds.
No ``swindle'' is needed here.

\section{Proof of the Main Theorem}
\label{sec:proofmainthm}

We first prove the Main Theorem using specific one-ended, open $4$-manifolds.
Then, we describe various ways of adapting the proof to other one-ended, open manifolds.
Let $T^k=\times_k S^1$ be the $k$-torus.
Define
\begin{align*}
M &= \sst{T^3} \cup_{\partial} (T^2 \times D^2)\\
N &= ([0,\infty) \times (S^1\times S^2)) \cup_{\partial} (S^1\times D^3)
\end{align*} 
So, $M$ is the surgered stringer based on $T^3$ capped with $T^2 \times D^2$, and
$N$ is the stringer based on $S^1\times S^2$ capped with $S^1\times D^3$ as in Figure~\ref{fig:main_examples}.
\begin{figure}[htbp!]
	\centerline{\includegraphics[scale=1]{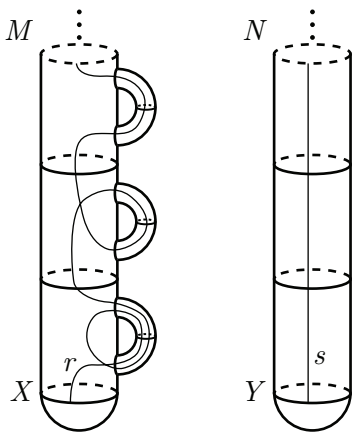}}
	\caption{Open manifolds $M$ and $N$ with rays $r\subset\Int{M}$ and $s\subset\Int{N}$.}
	\label{fig:main_examples}
\end{figure}

Let $\alpha = \sum_{i=0}^{\infty} a_i x^i \in\Z[[x]]$.
By Proposition~\ref{charclassss}, there is a ray $r\subset\Int{M}$ such that
\[
\br{r}^*_e=\pa{1,\ec{\alpha}}\in\Z \oplus \Z[[x]]/\Z[x] \cong\zrci{3}{M}
\]
As $N$ is one-ended, collared at infinity, and has dimension at least four, it contains a unique ray up to ambient isotopy.
So, let $s\subset\Int{N}$ be a straight ray as in Figure~\ref{fig:main_examples}.
By Example~\ref{stringer_example}, we have
\[
\br{s}^*_e = 1 \in \Z \cong\zrci{3}{N}
\]
Let $S=(M,r)\csi(N,s)$.
By Section~\ref{sec:stringersladders} and Theorem~\ref{thm:King}, we have $\zrci{*}{S} \cong \mathcal{A}$ where
\[
	\zrci{k}{S} \cong \mathcal{A}_k :=
	\begin{cases}
	\begin{aligned}
		((\Z &\oplus \Z[[x]]/\Z[x])  \oplus \Z)/I &&\tn{if $k = 3$,}\\
		\Z^3 											 &\oplus 0 \oplus \Z    &&\tn{if $k = 2$,}\\
		\Z^3											 &\oplus \Z[[\tau]]/\Z[\tau] \oplus \Z    &&\tn{if $k = 1$,}\\
		0 												 &\oplus  0 \oplus 0    &&\tn{otherwise}
	\end{aligned}
	\end{cases}
\]
Here, $I$ is the homogeneous ideal of degree $3$ generated by $\pa{\pa{1,\ec{\alpha}},-1}$.\\

Next, let $\mathcal{C}$ be an arbitrary graded $\Z$-algebra (possibly non-unital). We assign a sequence to $\mathcal{C}$ by the following procedure:
\begin{enumerate}[label=(\arabic*)]\setcounter{enumi}{0}
\item Let $J\leq \mathcal{C}_1$ be the subgroup of elements in $\mathcal{C}_1$ that are sent to $0$ by every element of the dual $\hom{\Z}{\mathcal{C}_1}{\Z}$ of $\mathcal{C}_1$. If $J$ is not a two-sided ideal of $\mathcal{C}$, then
return the empty sequence $()$ and end the procedure.
Otherwise, $J$ is a two-sided homogeneous ideal of $\mathcal{C}$.
\item Let $\mathcal{D}=\mathcal{C}/J$, which is a graded $\Z$-algebra.
\item In $\mathcal{D}$, let $U\leq\mathcal{D}_3$ be the subgroup generated by all products of three elements of degree one.
\item In $\mathcal{D}$, let $V \leq \mathcal{D}_3$ be the subgroup generated by elements that are a product of a degree one element and a degree two element but are not a product of three degree one elements.
\item If $V$ is not infinite cyclic, then return the empty sequence $()$ and end the procedure.
Otherwise, let $v$ be either generator of $V$.
\item Let $\pi:\mathcal{D}_3\to\mathcal{D}_3/U$ be the canonical homomorphism.
\item Return the height of $\pi(v)$ in $\mathcal{D}_3/U$ and end the procedure.
\end{enumerate}
The notion of the height of an element in an abelian group is reviewed below in Appendix~\ref{app:height}.
Let $h\pa{\mathcal{C}}$ denote the sequence (empty or infinite in length) determined by the procedure.

\begin{proposition}
If $\mathcal{C}$ and $\mathcal{C'}$ are isomorphic as graded $\Z$-algebras, then
$h\pa{\mathcal{C}}=h\pa{\mathcal{C'}}$.
Applied to the specific case of the end-cohomology algebra of the end-sum $S$ described above,
this yields $h\pa{\zrci{\ast}{S}}=h\pa{\mathcal{A}}$ which equals the height of $\ec{\alpha}\in\Z[[x]]/\Z[x]$.
\end{proposition}

\begin{proof}
Let $f:\mathcal{C}\to\mathcal{C'}$ be a graded $\Z$-algebra isomorphism.
Note that $f$ respects gradings, products, sums, and the $\Z$-module structure.
In particular, $f$ restricts to an isomorphism $\left.f\right|:\mathcal{C}_1\to\mathcal{C}'_1$,
and the latter induces the isomorphism of dual modules
$\hom{\Z}{\mathcal{C}'_1}{\Z}\to\hom{\Z}{\mathcal{C}_1}{\Z}$ given by $\psi\mapsto\psi\circ \left.f\right|$.
It follows that $f(J)=J'$.
If $J$ is not a two-sided ideal of $\mathcal{C}$, then $J'$ is not a two-sided ideal of $\mathcal{C}'$ and 
we have $h\pa{\mathcal{C}}=()=h\pa{\mathcal{C'}}$.
Otherwise, both $J$ and $J'$ are two-sided homogeneous ideals and
$f$ induces the graded $\Z$-algebra isomorphism
$F:\mathcal{D}\to\mathcal{D}'$ given by $F(x+J)=f(x)+J'$.
As $F$ respects gradings, products, sums, and the $\Z$-module structure, we get that
$F$ restricts to isomorphisms $\mathcal{D}_3\to\mathcal{D}'_3$, $U\to U'$, and $V\to V'$.
So, if $V$ is not infinite cyclic, then neither is $V'$ and we have $h\pa{\mathcal{C}}=()=h\pa{\mathcal{C'}}$.
Otherwise, $V$ and $V'$ are both infinite cyclic and $F(v)=\pm v'$.
The isomorphism $F$ induces the isomorphism
$G:\mathcal{D}_3/U \to \mathcal{D}'_3/U'$ given by $G(x+U)=F(x)+U'$.
So, the following diagram commutes
\begin{equation}\label{eq:heightsquare}\begin{split}
\xymatrix{
	\mathcal{D}_3 \ar[r]^-{\left.F\right|}_-{\cong} \ar[d]^-{\pi} & \mathcal{D}'_3 \ar[d]^-{\pi'}\\
	\mathcal{D}_3/U \ar[r]^-{G}_-{\cong}  & \mathcal{D}'_3/U'
}
\end{split}
\end{equation}
In particular, $G$ sends $\pi(v)\mapsto\pm\pi'(v')$.
The sequence $h(\mathcal{C})$ is the height of $\pi(v)\in\mathcal{D}_3/U$, and
the sequence $h(\mathcal{C}')$ is the height of $\pi'(v')\in\mathcal{D}'_3/U'$.
These sequences are equal since height is invariant under isomorphism
(see Corollary~\ref{heightiso} in the appendix below) and sign change.
This proves the first claim in the proposition and implies that
$h\pa{\zrci{\ast}{S}}=h\pa{\mathcal{A}}$. It remains to show that
$h\pa{\mathcal{A}}$ equals the height of $\ec{\alpha}\in\Z[[x]]/\Z[x]$.\\

Applying the procedure to $\mathcal{A}$, Corollary~\ref{dualpowermodpoly}
implies that
\[
J=0\oplus \Z[[\tau]]/\Z[\tau] \oplus 0 \leq \mathcal{A}_1
\]
Recall the product structure of $\mathcal{A}$ given in Section~\ref{sec:stringersladders}.
It implies that the product (in either order) of any element of $J$ with any 
element of $\mathcal{A}$ vanishes. So, $J$ is a two-sided homogeneous ideal of $\mathcal{A}$.
Taking the quotient of $\mathcal{A}$ by $J$, we obtain the algebra $\mathcal{D}$ where
\[
\mathcal{D}_k :=
	\begin{cases}
	\begin{aligned}
		((\Z &\oplus \Z[[x]]/\Z[x])  \oplus \Z)/I &&\tn{if $k = 3$,}\\
		\Z^3 											 &\oplus 0 \oplus \Z    &&\tn{if $k = 2$,}\\
		\Z^3											 &\oplus 0 \oplus \Z    &&\tn{if $k = 1$,}\\
		0 												 &\oplus  0 \oplus 0    &&\tn{otherwise}
	\end{aligned}
	\end{cases}
\]
Define $U \leq \mathcal{D}_3$ and $V\leq \mathcal{D}_3$ as in the procedure.
Using the well-known $\Z$-cohomology rings of $T^3$ and $S^1\times S^2$
(see~Hatcher~\cite[p.~216]{hatcher}),
we have $U = ((\Z\oplus 0)\oplus 0)/I\cong\Z$ and $V= ((0\oplus0)\oplus\Z)/I \cong \Z$ 
which are both infinite cyclic.
Let $v$ be either generator of $V$.
Let $\pi:\mathcal{D}_3\to \mathcal{D}_3/U$ be the canonical homomorphism.
Then, $h(\mathcal{A})$ equals the height of $\pi(v)\in\mathcal{D}_3/U$.\\

Conceptually, $\mathcal{A}_3=\mathcal{D}_3$ is obtained from $\Z\oplus\Z[[x]]/\Z[x]$ by summing with $\Z$ and then identifying the new $\Z$ with an infinite cyclic subgroup of
$\Z\oplus\Z[[x]]/\Z[x]$; essentially, this doesn't change the group.
More precisely, consider the homomorphism $\eta : \mathcal{D}_3\to \Z\oplus\Z[[x]]/\Z[x]$
defined by $\ec{((i,\ec{\beta}),j)} \mapsto (i+j,\ec{\beta}+j\ec{\alpha})$.
Noting that $\ec{((i,\ec{\beta}),j)}=\ec{((i+j,\ec{\beta}+j\ec{\alpha}),0)}$,
we see that $\eta$ is an isomorphism of groups.
Observe that $\eta(U)=\Z\oplus0$, and $\eta(V)$ is the infinite cyclic subgroup generated by
$\pa{1,\ec{\alpha}}$.
We have the commutative diagram
\begin{equation}\label{eq:heightsquaretail}\begin{split}
\xymatrix{
	\mathcal{D}_3 \ar[r]^-{\eta}_-{\cong} \ar[d]^-{\pi} & \Z\oplus\Z[[x]]/\Z[x] \ar[d] \ar[dr]\\
	\mathcal{D}_3/U \ar[r]^-{E}_-{\cong}  & (\Z\oplus\Z[[x]]/\Z[x])/(\Z\oplus 0) \ar[r]_-{\cong} & \Z[[x]]/\Z[x]
}
\end{split}
\end{equation}
where $E$ is induced by $\eta$ and the two downward homomorphisms in the triangle are the canonical projections
which induce the horizontal isomorphism.
At the level of elements, we have
\begin{equation}\label{eq:heightsquaretailelt}\begin{split}
\xymatrix{
	v \ar@{|->}[r] \ar@{|->}[d] & \pm\pa{1,\ec{\alpha}} \ar@{|->}[d] \ar@{|->}[dr]\\
	\pi(v) \ar@{|->}[r]  & \ec{\pm\pa{1,\ec{\alpha}}} \ar@{|->}[r] & \pm\ec{\alpha}
}
\end{split}
\end{equation}
As height is invariant under isomorphism and sign change,
the height of $\pi(v)\in\mathcal{D}_3/U$ equals the height of
$\ec{\alpha}\in\Z[[x]]/\Z[x]$ proving the proposition.
\end{proof}

As there exist uncountably many heights of elements in $\Z[[x]]/\Z[x]$ (see Lemma~\ref{heightrealize}), the Main Theorem is proved.
Crucial to our proof was the detection of the specific subgroups $U$ and $V$ in an isomorphically invariant manner.
To enable this, we chose manifolds with useful cup product structures.
In the absence of such cup products,
results at the end of Section~\ref{sec:stringersladders} above show that these subgroups cannot be so detected.\\

We close this section with a sample of variations of our proof of the Main Theorem.
Always, we consider a pair of one-ended, open $m$-manifolds $M$ and $N$.
\noindent
\begin{enumerate}[label=(\arabic*),leftmargin=*]\setcounter{enumi}{0}
\item To prove the Main Theorem for each dimension $m\geq 5$, consider the manifolds
\begin{align*}
M &= \sst{S^2\times S^{m-3}} \cup_{\partial} (S^2 \times D^{m-2})\\
N &= ([0,\infty) \times (S^1\times S^{m-2})) \cup_{\partial} (S^1\times D^{m-1})
\end{align*}
The proof is the same, except we consider the subgroups $U$ and $V$ of
\[
\mathcal{D}_{m-1} \cong ((\Z\oplus \Z[[x]]/\Z[x]) \oplus \Z)/I
\]
where $U = ((\Z\oplus 0)\oplus 0)/I\cong\Z$ is the subgroup generated by elements
that are a product of a degree two element and a degree $m-3$ element
(if $m=5$, then this means the product of two degree two elements),
and $V= ((0\oplus0)\oplus\Z)/I \cong \Z$ is the subgroup generated by elements that are a product of a degree one element
and a degree $m-2$ element.
\item For dimension $m=3$, consider closed, oriented surfaces $\Sigma_{g_1}$ and $\Sigma_{g_2}$
of distinct positive genera $g_1>g_2$ (the case $g_1<g_2$ may be handled similarly). Consider the manifolds
\begin{align*}
M &= \sst{\Sigma_{g_1}} \cup_{\partial} H_{g_1}\\
N &= ([0,\infty) \times \Sigma_{g_2}) \cup_{\partial} H_{g_2}
\end{align*}
where boundaries are capped with handlebodies.
As in the main proof, we let $S=(M,r)\csi(N,s)$ and,
by Section~\ref{sec:lbos} and Theorem~\ref{thm:King}, we have
\[
	\mathcal{D}_k :=
	\begin{cases}
	\begin{aligned}
		((\Z  &\oplus \Z[[x]]/\Z[x])  \oplus \Z)/I &&\tn{if $k = 2$,}\\
		\Z^{2g_1}                  &\oplus 0 \oplus\Z^{2g_2} &&\tn{if $k = 1$,}\\
		0 												 &\oplus  0  \oplus0   &&\tn{otherwise}
	\end{aligned}
	\end{cases}
\]
In this proof, we assume $\ec{\alpha}\ne0$.
So, products of degree one elements form a rank two subgroup of $\mathcal{D}_2$;
products in the first factor generate $U:=((\Z\oplus0)\oplus0)/I$, and
products in the third factor generate $V$ which is the $\Z$-span of $\ec{\pa{\pa{1,\ec{\alpha}},-1}}$.\\

Consider subgroups $C\leq \mathcal{D}_1$ such that: (1) products of elements of $C$ generate a rank one subgroup $D\leq\mathcal{D}_2$, and (2) the product of each element of $C$ with each element of $\mathcal{D}_1$ lies in $D$. Note that $\Z^{2g_1}\oplus0\oplus0$ is such a subgroup.
Among all of these subgroups, consider one $C'$ of maximal rank.
Suppose that $C'$ is not contained in $\Z^{2g_1}\oplus0\oplus0$.
Then, using Lemma~\ref{ranklemma}, we see that $C'$ meets both $\Z^{2g_1}\oplus0\oplus0$ and
$0\oplus0\oplus\Z^{2g_2}$ nontrivially. By Poincar\'{e} duality, elements that are the product of an element of $C'$ and an element of $\mathcal{D}_1$ generate a rank two subgroup of $\mathcal{D}_2$. This contradicts the defining properties of $C'$.
Therefore, $C'$ is contained in $\Z^{2g_1}\oplus0\oplus0$.
Among all such subgroups of maximal rank, $\Z^{2g_1}\oplus0\oplus0$ is maximal with respect to containment. This algebraically distinguishes $\Z^{2g_1}\oplus0\oplus0$ in an isomorphically invariant manner, and, hence, does the same for $U$. The rest of the proof is unchanged.
\item Similarly, one may prove the Main Theorem in each dimension $m\geq3$ using ladder manifolds in place of surgered stringers. Details are left to the interested reader.
\item In all of the above manifolds used to prove the Main Theorem in some dimension $m$,
one may use any compact caps to eliminate boundary and the proof is unchanged.
\end{enumerate}

\appendix

\section{Infinitely Generated Abelian Group Theory}\label{sec:igagt}

We present some relevant results from the theory of infinitely generated abelian groups.
This theory is subtle, beautiful, and (in our experience) not widely known among topologists.
For that reason, we provide proofs, as elementary as possible, for a few foundational results.
These results are then applied to prove a few propositions tailored specifically to our needs in this paper.
We close this appendix with a discussion of \emph{height} in an abelian group.\\

\subsection{Classical Results}

The additive abelian group $\Z[[x]]\cong\Z\times\Z\times\cdots$ is called the \emph{Baer-Specker group}.
Famously, it is not a free $\Z$-module; we include a proof of this fact below (see also Schr\"{o}er~\cite{schroeer}).
The additive abelian group $\Z[x]\cong\Z\oplus\Z\oplus\cdots$ is, of course, a free $\Z$-module with basis $\cpa{1,x,x^2,\ldots}$.
Throughout this and the next appendix, all maps are $\Z$-module homomorphisms.
The \emph{dual module} of a $\Z$-module $M$ is the $\Z$-module
\[
	M^{\ast}:=\hom{\Z}{M}{\Z}
\]

\begin{fact}\label{factpowermodpoly}
If $f:\Z[[x]]\to\Z$ vanishes on $\Z[x]$, then $f=0$.
\end{fact}

\begin{proof}
Fix an integer $p>1$.
Consider the element
\[
\alpha=a_0p^0x^0+a_1p^1x^1+a_2p^2x^2+\cdots\in\Z[[x]]
\]
where each $a_i\in\Z$. As $f$ vanishes on $\Z[x]$, we get
\[
f(\alpha)=f(a_kp^kx^k+a_{k+1}p^{k+1}x^{k+1}+\cdots)=p^k f(a_kp^0x^k+a_{k+1}p^{1}x^{k+1}+\cdots)
\]
and so $p^k$ divides $f(\alpha)$ for each integer $k>0$. Hence, $f(\alpha)=0$.\\

Next, fix coprime integers $p,q>1$.
Let $\gamma=\sum c_ix^i$ be an arbitrary element of $\Z[[x]]$.
We wish to write $\gamma=\alpha + \beta$ where $\alpha=\sum a_ip^ix^i$ and $\beta=\sum b_iq^ix^i$.
For each $i\geq0$, we seek integers $a_i$ and $b_i$ such that $c_i=a_ip^i+b_iq^i$, which is always possible since $p^i$ and $q^i$ are coprime. Now $f(\gamma)=f(\alpha)+f(\beta)=0+0=0$.
\end{proof}

\begin{corollary}\label{dualpowermodpoly}
$\pa{\Z[[x]]/\Z[x]}^{\ast} = \{0\}$.
\end{corollary}

\begin{proof}
Let $f:\Z[[x]]/\Z[x]\to\Z$, and let $\pi:\Z[[x]]\twoheadrightarrow\Z[[x]]/\Z[x]$ be the canonical surjection.
So, $f\circ \pi$ vanishes on $\Z[x]$. By Fact~\ref{factpowermodpoly}, $f\circ \pi=0$. As $\pi$ is surjective, $f=0$.
\end{proof}

Recall that $\tn{Hom}_{\Z}(-,\Z)$ distributes over finite direct sums.

\begin{corollary}
The $\Z$-module $\Z[[x]]/\Z[x]$ is uncountable and torsion-free, but does not split off $\Z$ as a direct summand
and is not a free $\Z$-module.
\end{corollary}

\begin{proof}
The first two claims are simple exercises. For the last two claims, otherwise $\pa{\Z[[x]]/\Z[x]}^{\ast}$ would be nonzero, contradicting Corollary~\ref{dualpowermodpoly}. 
\end{proof}

Intuitively, $\Z[[x]]/\Z[x]$ is flexible and large regarding injective maps of free $\Z$-modules into it, but is rigid regarding maps to free $\Z$-modules.

\begin{corollary}\label{wedgedual}
$\pa{\Z\oplus\Z[[x]]/\Z[x]\oplus\Z}^{\ast}\cong \Z^2$.
\end{corollary}

\begin{proof}
Immediate by Corollary~\ref{dualpowermodpoly}.
\end{proof}

Fact~\ref{factpowermodpoly} also implies that two maps $\Z[[x]]\to\Z$ that agree on $\Z[x]$ must be equal (consider their difference).
Combining this observation with projections, we see that two maps $\Z[[x]]\to\Z[[x]]$ that agree on $\Z[x]$ must be equal (see also Fuchs~\cite[Lemma~94.1]{fuchsinfag}).
We mention that $\Z[x]^{\ast}\cong\Z[[x]]$ since $\Z[x]$ is free with $\Z$-basis $\cpa{1,x,x^2,\ldots}$.
The isomorphism is $f\mapsto \sum_{i\geq0} f(x^i) x^i$.
For completeness, we prove the ``dual'' fact that $\Z[[x]]^{\ast}\cong\Z[x]$.\\

The following fact is long known.

\begin{fact}\label{factpowertopoly}
If $f:\Z[[x]]\to\Z$, then $f(x^k)=0$ for cofinitely many $k$.
\end{fact}

\begin{proof}
Consider an element
\[
	\alpha = a_0!+a_1!x+a_2!x^2+\cdots\in\Z[[x]]
\]
for integers $0<a_0<a_1<a_2<\cdots$ to be determined. For each $k$, we have a tail of $\alpha$ denoted
\[
	a_k!\beta_k = a_k!x^k+a_{k+1}!x^{k+1}+\cdots \in\Z[[x]]
\]
where
\[
\beta_k=\frac{a_k!}{a_k!}x^k+\frac{a_{k+1}!}{a_k!}x^{k+1}+\cdots \in\Z[[x]]
\]
Notice that $\beta_k$ lies in $\Z[[x]]$ since the $a_j$'s are increasing.
We have
\[
f(\alpha)=a_0!f(1)+a_1!f(x)+\cdots+a_k!f(x^k)+a_{k+1}!f(\beta_{k+1})\in\Z
\]
Hence
\[
a_{k+1}!f(\beta_{k+1})=f(\alpha)-a_0!f(1)-a_1!f(x)-\cdots-a_k!f(x^k)
\]
By the triangle inequality
\[
a_{k+1}!\card{f(\beta_{k+1})}\leq\card{f(\alpha)}+a_0!\card{f(1)}+a_1!\card{f(x)}+\cdots+a_k!\card{f(x^k)}
\]
Therefore
\[
\card{f(\beta_{k+1})}\leq \frac{\card{f(\alpha)}}{a_{k+1}!}+ \frac{a_0!\card{f(1)}+a_1!\card{f(x)}+\cdots+a_k!\card{f(x^k)}}{a_{k+1}!}
\]
The first term on the right side tends to zero as $k\to\infty$, and we may inductively choose the positive integers $a_j$ so that the second term is less than $1/(k+1)$.
Hence, the nonnegative integers $\card{f(\beta_{k+1})}$ tend to $0$ as $k\to\infty$. So, $f(\beta_{k+1})=0$ for cofinitely many $k$.\\

Now, $a_k!\beta_k-a_{k+1}!\beta_{k+1}=a_k!x^k$ and so $f(x^k)=0$ for cofinitely many $k$, as desired. 
\end{proof}

\begin{corollary}\label{powerdual}
$\Z[[x]]^{\ast}\cong\Z[x]$.
\end{corollary}

\begin{proof}
Let $f:\Z[[x]]\to\Z$. Fact~\ref{factpowertopoly} implies that $f(x^k)=0$ for cofinitely many $k$.
Thus, $g:\Z[[x]]\to\Z$ defined by
\[
g\pa{\sum_{i\geq0} c_i x^i} = \sum_{i\geq0} c_i f(x^i) x^i
\]
is a well-defined $\Z$-module homomorphism.
As $f$ and $g$ agree on $\Z[x]$, we see that $f=g$.
Note that the isomorphism $\Z[[x]]^{\ast}\to\Z[x]$ is $f\mapsto \sum_{i\geq0} f(x^i) x^i$.
\end{proof}

\begin{corollary}
$\Z[[x]]$ is not a free $\Z$-module.
\end{corollary}

\begin{proof}
Otherwise, there is a $\Z$-basis for $\Z[[x]]$ which is necessarily uncountable and so $\Z[[x]]^{\ast}$ is uncountable.
This contradicts Corollary~\ref{powerdual}.
\end{proof}

\subsection{Applications of Duals to End-Cohomology of Ladder Manifolds}\label{appclm}

We now move to $\Z$-modules and algebras arising from ladder manifolds.
Let $\la{X}{Y}$ be a ladder manifold based on closed, connected, and oriented $n$-manifolds $X$ and $Y$ where $n\geq 2$.
Recall from Section~\ref{sec:stringersladders} that the degree $n$ subgroup of the end-cohomology algebra of $\la{X}{Y}$ is
\[
\zrci{n}{\la{X}{Y}} \cong (\Z \oplus \Z[[\sigma]] \oplus \Z) / K
\]
where
$K = \cpa{\left. \pa{\sum \beta_i, \beta, -\sum \beta_i} \right| \beta = \sum \beta_i \sigma^i \in \Z[\sigma]} \cong \Z[\sigma]$.
Sometimes, we write $x$ instead of $\sigma$ in $K$.
We begin by computing the dual module of $(\Z \oplus \Z[[x]] \oplus \Z) / K$.

\begin{lemma}\label{vanish_lemma}
If $h:\Z[[x]]\to\Z$ vanishes on
\[
L:=\cpa{\left. \beta=\sum \beta_i x^i \in \Z[x] \right| \sum \beta_i=0}
\]
then $h=0$.
\end{lemma}

\begin{proof}
By Fact~\ref{factpowertopoly}, there exists $k\geq0$ such that $h(x^k)=0$.
Let $\alpha=\sum a_i x^i$ be an arbitrary element of $\Z[x]$ and let $n=\sum a_i$.
Then, $\alpha-nx^k\in L$ and $0=h(\alpha-nx^k)=h(\alpha)$.
So, $h$ vanishes on $\Z[x]$. By Fact~\ref{factpowermodpoly}, $h=0$.
\end{proof}

\begin{corollary}
If $f:\Z\oplus\Z[[x]]\oplus\Z\to\Z$ vanishes on $K$, then $f=0$ on $0\oplus\Z[[x]]\oplus0$ and $f(r,\gamma,s)=j(r+s)$ for some fixed integer $j$.
\end{corollary}

\begin{proof}
Consider the inclusion $i:\Z[[x]]\hookrightarrow\Z\oplus\Z[[x]]\oplus\Z$ given by $i(\gamma)=(0,\gamma,0)$.
The composition $f\circ i: \Z[[x]]\to\Z$ vanishes on $L$. By Lemma~\ref{vanish_lemma}, $f\circ i=0$.
Thus, $f$ vanishes on $0\oplus\Z[[x]]\oplus0$. As $(1,1,-1)\in K$, we get
\[
0=f(1,1,-1)=f(1,0,0)+f(0,1,0)-f(0,0,1)=f(1,0,0)-f(0,0,1)
\]
and $f(1,0,0)=f(0,0,1)$. Define $j:=f(1,0,0)$. Then, $f(r,\gamma,s)=j(r+s)$.
\end{proof}

\begin{corollary}\label{modKdual}
$\pa{(\Z\oplus\Z[[x]]\oplus\Z)/K}^{\ast}\cong \Z$.
\end{corollary}

\begin{proof}
Let $f:(\Z\oplus\Z[[x]]\oplus\Z)/K \to \Z$,
and let $\pi:\Z\oplus\Z[[x]]\oplus\Z\twoheadrightarrow (\Z\oplus\Z[[x]]\oplus\Z)/K$ be the canonical surjection.
The composition $f\circ \pi$ vanishes on $K$. By the previous corollary, $f\circ \pi(r,\gamma,s)=j(r+s)$ for some fixed integer $j$. As $\pi$ is surjective, $f(\ec{(r,\gamma,s)})=j(r+s)$. In particular, $j=1$ gives a generator for the dual module in question.
\end{proof}

\begin{corollary}\label{nonisoduals}
The uncountable, torsion-free $\Z$-modules $\Z\oplus\Z[[x]]/\Z[x]\oplus\Z$ and $(\Z\oplus\Z[[x]]\oplus\Z)/K$ are not isomorphic.
\end{corollary}

\begin{proof}
They have nonisomorphic duals by Corollaries~\ref{wedgedual} and~\ref{modKdual}.
\end{proof}

As an application of Corollary~\ref{nonisoduals}, consider the space in Figure~\ref{fig:ws}.
\begin{figure}[htbp!]
    \centerline{\includegraphics[scale=1.0]{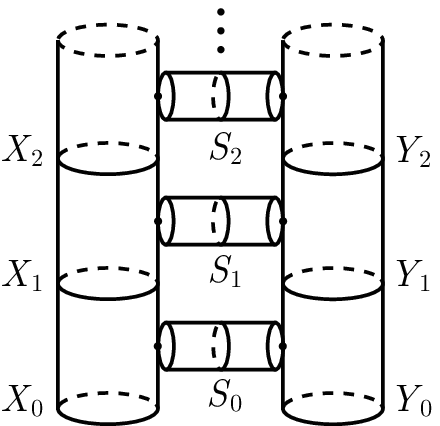}}
    \caption{Wedge space $\ws{X}{Y}$.}
\label{fig:ws}
\end{figure}
The \emph{wedge space} $\ws{X}{Y}$ based on $X$ and $Y$ (a nonmanifold) is obtained from the disjoint union of the stringers on $X$ and $Y$ by simply wedging on the rungs as shown.
The end-cohomology algebra of $\ws{X}{Y}$ is
\[
	\rci{k}{\ws{X}{Y}}{R} \cong
	\begin{cases}
		\cohom{n}{X}{R} \oplus R[[\sigma]]/R[\sigma] \oplus \cohom{n}{Y}{R} &\tn{if $k = n$,}\\
		\cohom{k}{X}{R} \oplus 0 \oplus \cohom{k}{Y}{R} &\tn{if $2\leq k \leq n-1$,}\\
		\cohom{1}{X}{R} \oplus R[[\tau]]/R[\tau] \oplus \cohom{1}{Y}{R} &\tn{if $k = 1$,}\\		
		0 &\tn{otherwise}
	\end{cases}
\]
where the cup product is coordinatewise in the direct sum and vanishes in the middle coordinate.
While the end-cohomology algebra of the wedge space $\ws{X}{Y}$ bears a striking resemblance
to that of the ladder manifold $\la{X}{Y}$, they are not isomorphic.

\begin{corollary}\label{spaces_distinct}
The end-cohomology algebras of the ladder manifold $\la{X}{Y}$ and the wedge space $\ws{X}{Y}$ are not isomorphic.
In particular, these two spaces are not proper homotopy equivalent.
\end{corollary}

\begin{proof}
The degree $n$ subgroups of these two algebras are nonisomorphic by Corollary~\ref{nonisoduals}.
\end{proof}

\subsection{Height in Abelian Groups}\label{app:height}

The notion of the height of an element plays an important role in the study of infinite abelian groups.
Let $G$ be an additive abelian group, $g\in G$, and $p>1$ a prime integer.
Consider the following equation in $G$ for integers $k\geq0$
\begin{equation}\label{heighteq}
p^k x = g \tag{$\dag$}
\end{equation}
The \emph{height of $g\in G$ at $p$} is
\[
H_p(g):=k
\]
where $k\in\Z_{\geq0}$ is maximal such that~\eqref{heighteq} has a solution $x\in G$.
If~\eqref{heighteq} has a solution for every $k\in\Z_{\geq0}$, then we write $H_p(g)=\infty$.
Let $2=p_1<p_2<\cdots$ be the primes.
The \emph{height of $g\in G$} is the sequence
\begin{equation}\label{heightgen}
H(g):=(H_{p_1}(g),H_{p_2}(g),\ldots) \in \{0,1,2,\ldots,\infty \}^{\N} \tag{$\ddag$}
\end{equation}
For example, $H(0)=(\infty,\infty,\infty,\ldots)$ in every abelian group $G$.
If $G$ is a field of characteristic zero, then $H(g)=(\infty,\infty,\infty,\ldots)$ for all $g\in G$.
If $G$ is a field of characteristic $p$ and $0\ne g\in G$, then $H_p(g)=0$ and
$H_q(g)=\infty$ for each prime $q\ne p$.
In general, the height of $g$ depends on $G$, since the solutions $x$ of~\eqref{heighteq} are required to lie in $G$.
For the sake of intuition, it is useful to look at some examples.
In what follows, the partial ordering $\preceq$ on $\{0,1,2,\ldots,\infty \}^{\N}$ is defined by:
$(m_i)_{i\in \N} \preceq (n_i)_{i\in\N}$ if and only if $m_i\leq n_i$ for all $i\in\N$.

\begin{example}\label{exampleZ}
Let $G=\Z$ and $g=1400$. The prime factorization $1400=2^{3}5^{2}7$ reveals that
$H_2(g)=3$, $H_5(g)=2$, $H_7(g)=1$, and $H_p(g)=0$ for all other primes $p$.
Inserting this data into~\eqref{heightgen} give us
\[
H(1400)=(3,0,2,1,0,0,0,\ldots)
\]
More generally, if $0\ne g = \prod_{i=1}^{\infty} p_i^{k_i}$, then
\[
H(g)=(k_1,k_2,k_3,\ldots)
\]
where all $k_i<\infty$ and all but finitely many $k_i$ are zero.
\end{example}

\begin{example}\label{exampleS}
Let $G=\Z[x]$ and $\alpha=\sum_{i=0}^n a_i x^i\in\Z[x]$ where $a_n\ne0$. 
Recall that the \emph{content of $\alpha$} is
\[
c(\alpha)=\gcd\pa{a_0,a_1,\ldots,a_n}\in\Z_{>0}
\]
It is straightforward to verify that the height of $\alpha\in\Z[x]$ equals the 
height of $c(\alpha)\in\Z$.
Thus, as in the previous example, all heights of nonzero elements
contain finitely many nonzero entries each of which is finite.
\end{example}

\begin{example}\label{exampleP}
Let $G=\Z[[x]]$ and $0\ne\alpha=\sum_{i=0}^{\infty} a_i x^i\in\Z[[x]]$.
We define the \emph{content of $\alpha$} to be
\[
c(\alpha)=\gcd\pa{a_0,a_1,a_2,\ldots}\in\Z_{>0}
\]
This is well-defined since some $a_n\ne 0$ and
\[
\infty > \gcd(a_0,a_1,\ldots,a_n)\geq \gcd(a_0,a_1,\ldots,a_{n+1}) \geq \gcd(a_0,a_1,\ldots,a_{n+2}) \geq \cdots
\]
Only finitely many of these inequalities can be strict by the well-ordering principle.
In particular, $c(\alpha)=\gcd\pa{a_0,a_1,\ldots,a_N}$ for some nonnegative integer $N=N(\alpha)$.
Again, it is straightforward to verify that the height of $\alpha\in\Z[[x]]$ equals the 
height of $c(\alpha)\in\Z$, which yields exactly the same collection of heights as in the previous two examples.
Therefore, even though $\Z[[x]]$ is uncountable, its elements realize only countably many heights.
\end{example}

\begin{example}\label{height_ex}
In $G=\Z[[x]]/\Z[x]$ height becomes more interesting. Consider, for example
\begin{align*}
H\pa{\ecl{2+2x+2x^2+\cdots}}	&= (1,0,0,\ldots)\\
H\pa{\ecl{2^35^0+2^35^1x+2^35^2x^2+\cdots}}	&= (3,0,\infty,0,0,\ldots)\\
H\pa{\ecl{2^1x+2^23^2x^2+2^33^35^3x^3+\cdots}}	&= (\infty,\infty,\infty,\ldots)
\end{align*}
The key here is that, for any power series $\sum_{i=0}^{\infty} a_i x^i\in\Z[[x]]$ representing an
element of $\Z[[x]]/\Z[x]$, we may ignore any finite initial sum $\sum_{i=0}^{j} a_i x^i$.
This leads to the following.
\end{example}

\begin{lemma}\label{heightrealize}
Let $G=\Z[[x]]/\Z[x]$ and let $h=(h_1,h_2,\ldots)\in\{0,1,2,\ldots,\infty\}^{\N}$ be a height.
Then, there exists $g\in G$ such that $H(g)=h$.
\end{lemma}

\begin{proof}
The idea of the proof is contained in Example~\ref{height_ex}.
Let $2=p_1<p_2<\cdots$ be the rational primes.
For each integer $i\geq1$, define
\[
a_i:=p_1^{e(1,i)} p_2^{e(2,i)} \cdots p_i^{e(i,i)} \in \Z_{>0}
\]
where
\[
	e(n,i) :=
	\begin{cases}
		i &\tn{if $h_n\geq i$}\\
		h_n &\tn{if $h_n<i$}
	\end{cases}
\]
Let $g=\ecl{\sum_{i=1}^{\infty} a_i x^i} \in G$.
Then, the height of $g\in G$ is $H(g)=h$.
\end{proof}

For a general discussion of height, see Fuchs~\cite[Ch.~VII]{fuchsag}.
For our purposes, the crucial facts are Lemma~\ref{heightrealize} and the following.

\begin{lemma}
Let $\phi:G\to G'$ be a homomorphism of abelian groups. Then, for all $g\in G$ we have
$H(g)\preceq H(\phi(g))$.
\end{lemma}

\begin{proof}
Let $p>1$ be a prime integer and suppose $g=p^kx$ in $G$.
Then $\phi(g)=\phi(p^kx)=p^k\phi(x)$ in $G'$.
Therefore, $H_p(g)\leq H_p(\phi(g))$.
\end{proof}

\begin{corollary}\label{heightiso}
Let $\phi:G\to G'$ be an isomorphism of abelian groups. Then, for all $g\in G$ we have
$H(g)=H(\phi(g))$.
\end{corollary}

\end{document}